\documentclass[11pt,a4paper,twoside]{article}
\usepackage{amsthm, amsfonts,amsmath}

\newtheorem{Theorem}{Theorem}
\newtheorem{Corollary}{Corollary}
\newtheorem{Definition}{Definition}
\newtheorem{Example}{Example}
\newtheorem{Lemma}{Lemma}
\newtheorem{Proposition}{Proposition}
\newtheorem{Remark}{Remark}

\def\Ric{\operatorname{Ric}}
\def\tr{{\rm\,trace\,}}

\def\Div{\operatorname{div}}
\def\vol{\operatorname{vol}}

\title{Geometry of weak metric $f$-manifolds: a survey}

\author{Vladimir Rovenski
\footnote{Department of Mathematics, University of Haifa, Mount Carmel, 3498838 Haifa, Israel
\newline e-mail: {\tt vrovenski@univ.haifa.ac.il}
}
}

\begin{document}

\date{}

\maketitle

\begin{abstract}
A weak $f$-structure on a smooth manifold, introduced by the author and R. Wolak (2022), generalizes K. Yano's (1961) $f$-structure.
This generalization allows us to revisit classical theory and discover new applications related to Killing vector fields, totally geodesic foliations, Ricci-type solitons, and Einstein-type metrics.
This article reviews the results on weak metric $f$-manifolds, where the complex structure on the contact distribution of a metric $f$-structure is replaced with a nonsingular skew-symmetric tensor, and explores its distinguished classes.


\end{abstract}


\section{Introduction}

Contact geometry has garnered increasing interest due to its significant role in theoretical physics.
An important class of contact metric manifolds are ${K}$-contact manifolds (thus the structural vector field is Killing) with two subclasses: Sasakian and cosymplectic manifolds.
Every cosymplectic manifold is locally the product of
$\mathbb{R}$
and a K\"{a}hler manifold.
A~Riemannian manifold $(M^{2n+1},g)$ with contact 1-form $\eta$ is Sasakian if its Riemannian cone $M\times\mathbb{R}^{>0}$ with metric $t^2 g+dt^2$ is a K\"{a}hler manifold.
Recent research has been driven by the intriguing question of how Ricci solitons
-- self-similar solutions of the Ricci flow equation
--
can be significant for the geometry of contact metric manifolds.
Some studies have explored conditions under which a~contact metric manifold equipped with a Ricci-type soliton structure carries a canonical metric,
such as
an Einstein-type metric, e.g.,~\cite{G-D-2020,G-2023,Mikes-2016,Mikes-Hui-2016}.

An~$f$-structure introduced by K.\,Yano on a smooth manifold $M^{2n+s}$ serves as a higher-dimensio\-nal analog
of {almost complex structures} ($s=0$) and {almost contact structures} ($s=1$).
This structure is defined by a (1,1)-tensor field $f$ of rank $2n$ such that $f^3 + f = 0$, see \cite{fip,yan}.
The tangent bundle splits into two complementary subbundles:
 $TM=f(TM)\oplus\ker f$.
The~restriction of $f$ to the $2n$-dimensional distribution $f(TM)$ defines a {complex structure}.
The existence of the $f$-structure on $M^{2n+s}$ is equivalent to a reduction of the
structure group to $U(n)\times O(s)$, see~\cite{b1970}.
A submanifold $M$ of an almost complex manifold $(\bar M,J)$ that satisfies the condition $\dim(T_xM\cap J(T_xM))=const>0$
naturally possesses an $f$-structure, see~\cite{L-1969}.
An $f$-structure is a~special case of an {\em almost product structure}, defined by two complementary orthogonal distributions of a Riemannian mani\-fold $(M,g)$, with Naveira's 36 distinguished classes, see \cite{nav-1983}. Foliations appear when one or both distributions are involutive.
An~interesting case occurs when the subbundle $\ker f$ is parallelizable, leading to a framed $f$-structure
for which the reduced structure group is $U(n)\times {\rm id}_s$.
In this scenario, there exist vector fields $\{\xi_i\}_{1\le i\le s}$ spanning $\ker f$
with dual 1-forms $\{\eta^i\}_{1\le i\le s}$, satisfying ${f}^2 = -{\rm id} + \sum\nolimits_{\,i}{\eta^i}\otimes {\xi_i}$.
A~compatible metric exists on any framed $f$-manifold, e.g., \cite{b1970}:
$g({f}X,{f}Y)=g(X,Y)-\sum\nolimits_{\,i}{\eta^i}(X)\,{\eta^i}(Y)$,
and we~obtain the {metric $f$-structure},
see~\cite{BP-2016,CFF-1990,fip,tpw-2014,yan,YK-1985}.
Geometers have introduced various broad classes of metric $f$-structures.
A notable class is {Kenmotsu $f$-manifolds}, see~\cite{SV-2016} (Kenmotsu manifolds when $s=1$, see \cite{kenmotsu1972class}), characterized in terms of warped products of $\mathbb{R}^s$ and a K\"{a}hler manifold.
 A~metric $f$-manifold is termed a {${K}$-manifold} if it is normal and $d \Phi=0$, where $\Phi(X,Y)=g(X,{f} Y)$.
Two important subclasses of ${K}$-manifolds are {${\cal C}$-manifolds} if $d\eta^i=0$ and {${\cal S}$-manifolds} if $d\eta^i = \Phi$ for any~$i$,
see \cite{b1970}.
Omitting the normality condition, we get almost ${K}$-manifolds, almost ${\cal S}$-manifolds and almost ${\cal C}$-manifolds, respectively.
An~$f$-{K}-contact manifold
is a normal metric $f$-manifold, whose structure vector fields are unit Killing vector fields;
the structure is intermediate between metric $f$-structure and $S$-structure,
see \cite{Goertsches-2,Di-T-2006}.
The~distribution $\ker f$ of a ${K}$-manifold is tangent to a $\mathfrak{g}$-foliation with flat totally geodesic leaves.
Note that there are no Einstein metrics on $f$-{K}-contact~manifolds.

The Extrinsic Geometry is concerned with properties of submanifolds (as being totally geodesic) that depend on the second fundamental form,
which, roughly speaking, describes how the subma\-nifolds are located in the ambient Riemannian manifold.
The Extrinsic Geometry of foliations (i.e., involutive distributions) is a field of Riemannian geometry, which studies the properties expressed by the second fundamental tensor of the leaves. Although the Riemann tensor belongs to the intrinsic geometry, a special component called mixed sectional curvature is a part of Extrinsic Geometry of foliations.
Totally geodesic foliations have simple extrinsic geometry and appear on manifolds with degenerate tensor fields, see~\cite{Rov-Wa-2021}.
A~key problem posed in this context in \cite{fip} is to identify suitable structures on manifolds that lead to totally geodesic foliations.


In \cite{RWo-2}, we initiated the study of weak $f$-structures on a smooth $(2n+s)$-dimensional mani\-fold
(that is, the linear complex structure on the subbundle ${\cal D}=f(TM)$ of a metric $f$-structure is replaced with a nonsingular skew-symmetric tensor).
These generalize the metric $f$-structure (the weak almost contact metric structure for $s=1$, see
\cite{rst-55}) and its satellites, allow us to look at classical theory in a new way and find new applications of Killing vector fields, totally geodesic foliations, Einstein-type metrics and Ricci-type~solitons.

\smallskip

The article reviews the results of our works \cite{rst-43,Rov-splitting,rov-127,rst-57,Rov-Wa-2021,RWo-2} on the geometry of weak metric $f$-manifolds and their distinguished classes. It is organized as follows:
Section~\ref{sec:01} (following the Introduction) presents the basics of weak metric $f$-manifolds and introduces their important subclasses.
It also investigates the normality condition and derives the covariant derivative of $f$ with a new tensor ${\cal N}^{\,(5)}$
and shows that the distribution ${\cal D}^\bot$ of a weak almost ${\cal S}$-manifold and a weak almost ${\cal C}$-manifold
is tangent to a $\mathfrak{g}$-foliation with an abelian Lie~algebra.
Section~\ref{sec:03a} presents basics of weak almost ${\cal S}$-manifolds and shows that these manifolds are endowed with totally geodesic foliations.
Section~\ref{sec:3cs} discussed weak ${\cal C}$-structures
and shows that the weak metric $f$-structure is a weak ${\cal S}$-structure if and only if it is an ${\cal S}$-structure.
Section~\ref{sec:02} characterizes
weak $f$-{K}-contact manifolds among all weak almost ${\cal S}$-manifolds by the property ${f}=-\nabla{\xi_i}\ (1\le i\le s)$ known for $f$-{K}-contact manifolds, and presents sufficient conditions
under which a Riemannian mani\-fold endowed with a set of unit Killing vector fields is a weak $f$-{K}-contact manifold.
We find the Ricci curvature of a weak $f$-{K}-contact manifold in the directions $\ker f$ and show that the mixed sectional curvature is positive.
Using this, the weak $f$-{K}-contact structure can be  deformed to the $f$-{K}-contact structure
and obtain a topological obstruction (including the Adams number) to the existence of weak $f$-K-contact manifolds.
Then we show that there are no Einstein weak $f$-{K}-contact manifolds with $s>1$.
%
Section~\ref{sec:04} presents
sufficient conditions for a weak $f$-{K}-contact manifold with a generalized gradient Ricci soliton to be a quasi Einstein or Ricci flat manifold.
%
In~Section~\ref{sec:05}, we find sufficient conditions for
a compact weak $f$-{\rm K}-contact manifold with the $\eta$-Ricci structure of constant scalar curvature to be~$\eta$-Einstein.
Sections~\ref{sec:02-f-beta} and \ref{sec:03-f-beta} show that a weak $\beta$-Kenmotsu $f$-manifold is locally a twisted product of $\mathbb{R}^s$
and a weak K\"{a}hler manifold, and in the case of an additional $\eta$-Ricci soliton structure,
explore their potential to be $\eta$-Einstein~manifolds of constant scalar curvature.
%
The proofs of some results given in the article for the convenience of the reader use properties of the new tensors, as well as constructions needed in the classical case.

\section{Preliminaries}
\label{sec:01}

In this section, we review the basics of the weak metric $f$-structure,
see \cite{rst-43,RWo-2}.
First, let us gene\-ralize the notion of a framed $f$-structure
\cite{BP-2016,CFF-1990,gy-1970,tpw-2014,yan,YK-1985}
called
an $f$-structure with complemented frames in~\cite{b1970}
or, an $f$-structure with parallelizable kernel
in \cite{fip}.

\begin{Definition}\rm
A~\textit{framed weak $f$-structure} on a smooth manifold $M^{2n+s}\ (n,s>0)$ is a set $({f},Q,{\xi_i},{\eta^i})$, where
${f}$ is a $(1,1)$-tensor of rank $2\,n$, $Q$ is a nonsingular $(1,1)$-tensor,
${\xi_i}\ (1\le i\le s)$ are structure vector fields and ${\eta^i}\ (1\le i\le s)$ are 1-forms, satisfying
\begin{align}\label{2.1}
 {f}^2 = -Q +
 \sum\nolimits_{\,i}
 {\eta^i}\otimes {\xi_i},\quad
 {\eta^i}({\xi_j})=\delta^i_j,\quad
 Q\,{\xi_i} = {\xi_i}.
\end{align}
Then equality $f^3+fQ=0$ holds.
 If there exists a Riemannian metric $g$ on $M^{2n+s}$
 such~that
\begin{align}\label{2.2}
 g({f} X,{f} Y)= g(X,Q\,Y) -\sum\nolimits_{\,i}{\eta^i}(X)\,{\eta^i}(Y),\quad X,Y\in\mathfrak{X}_M,
\end{align}
then $({f},Q,{\xi_i},{\eta^i},g)$ is a {\it weak metric $f$-structure},
 and
$g$ is called a \textit{compatible} metric.
\end{Definition}

Assume that
a $2\,n$-dimensional contact distribution ${\cal D}=\bigcap_{\,i}\ker{\eta^i}$ is ${f}$-invariant.
Note that for the framed weak $f$-structure,
${\cal D}=f(TM)$ is true, and
\begin{align}\label{Eq-f-01}
 {f}\,{\xi_i}=0,\quad {\eta^i}\circ{f}=0,\quad \eta^i\circ Q=\eta^i,\quad [Q,\,{f}]=0 .
\end{align}
By the above,
the distribution ${\cal D}^\bot=\ker f$ is spanned by $\{\xi_1,\ldots,\xi_s\}$ and is invariant for $Q$.

\begin{Remark}\rm
The concept of an almost paracontact structure is analogous to the concept of an almost contact structure and is closely related to an almost product structure. A para-$f$-structure with $s=2$ arises on hypersurfaces in almost contact manifolds.
Similarly to \eqref{2.1}, we define a \textit{framed weak para}-$f$-\textit{structure} on $M^{2n+s}$, see more details in \cite{rov-121}, by
\[
 f^2 = Q -\sum\nolimits_{\,i}\eta^i\otimes\xi_i,\quad
 {\eta^i}({\xi_j})=\delta^i_j,\quad
 Q\,{\xi_i} = {\xi_i},
\]
and
assume that a $2\,n$-dimensional contact distribution ${\cal D}=\bigcap_{\,i}\ker{\eta^i}$ is ${f}$-invariant.
\end{Remark}

The framed weak $f$-structure
is called {\it normal} if the following tensor is~zero:
\begin{align}\label{2.6X}
 {\cal N}^{\,(1)} = [{f},{f}] + 2\sum\nolimits_{\,i}d{\eta^i}\otimes{\xi_i}.
\end{align}
The Nijenhuis torsion
of a (1,1)-tensor ${S}$ and the exterior derivative of a 1-form ${\omega}$ are given~by
\begin{align*}
 & [{S},{S}](X,Y) = {S}^2 [X,Y] + [{S} X, {S} Y] - {S}[{S} X,Y] - {S}[X,{S} Y],\quad X,Y\in\mathfrak{X}_M, \\
 & d\omega(X,Y) = \frac12\,\{X({\omega}(Y)) - Y({\omega}(X)) - {\omega}([X,Y])\},\quad X,Y\in\mathfrak{X}_M.
\end{align*}
Using the Levi-Civita connection $\nabla$ of $g$, one can rewrite $[S,S]$ as
\begin{align}\label{4.NN}
 [{S},{S}](X,Y) = ({S}\nabla_Y{S} - \nabla_{{S} Y}{S}) X - ({S}\nabla_X{S} - \nabla_{{S} X}{S}) Y .
\end{align}
The following tensors
${\cal N}^{\,(2)}_i, {\cal N}^{\,(3)}_i$ and ${\cal N}^{\,(4)}_{ij}$
on framed weak $f$-manifolds, see \cite{rst-43, rov-127}, are well known in the classical theory, see \cite{b1970}:
\begin{align*}
 {\cal N}^{\,(2)}_i(X,Y) &= (\pounds_{{f} X}\,{\eta^i})(Y) - (\pounds_{{f} Y}\,{\eta^i})(X)
 =2\,d{\eta^i}({f} X,Y) - 2\,d{\eta^i}({f} Y,X) ,  \\
 {\cal N}^{\,(3)}_i(X) &= (\pounds_{{\xi_i}}{f})X
 = [{\xi_i}, {f} X] - {f} [{\xi_i}, X],\\
 {\cal N}^{\,(4)}_{ij}(X) &= (\pounds_{{\xi_i}}\,{\eta^j})(X)
 = {\xi_i}({\eta^j}(X)) - {\eta^j}([{\xi_i}, X])
 = 2\,d{\eta^j}({\xi_i}, X) .
\end{align*}

\begin{Remark}
\rm
Let $M^{2n+s}(f,Q,\xi_i,\eta^i)$ be a framed weak $f$-manifold.
Consider the product manifold $\bar M = M^{2n+s}\times\mathbb{R}^s$,
where $\mathbb{R}^s$ is a Euclidean space with a basis $\partial_1,\ldots,\partial_s$,
and define tensors $J$ and $\bar Q$ on $\bar M$ putting
 $J(X, \sum\nolimits_{\,i}a^i\partial_i) = (fX - \sum\nolimits_{\,i}a^i\xi_i,\, \sum\nolimits_{\,j}\eta^j(X)\partial_j)$
 and
 $\bar Q(X, \sum\nolimits_{\,i}a^i\partial_i) = (QX,\, \sum\nolimits_{\,i}a^i\partial_i)$
for $a_i\in C^\infty(M)$.
It can be shown that $J^{\,2}=-\bar Q$.
 The~tensors ${\cal N}^{\,(2)}_i, {\cal N}^{\,(3)}_i, {\cal N}^{\,(4)}_{ij}$ appear when~we derive the integrability condition $[J, J]=0$
and express the normality condition ${\cal N}^{\,(1)}=0$ for~$(f,Q,\xi_i,\eta^i)$.
\end{Remark}

A~fra\-med weak $f$-manifold admits a compatible metric if ${f}$
has a skew-sym\-metric representation, i.e., for any $x\in M$ there exist a~frame $\{e_i\}$ on a neighborhood $U_x\subset M$,
for which ${f}$ has a skew-symmetric matrix, see \cite{RWo-2}.
For a weak metric $f$-manifold, the tensor ${f}$ is skew-symmetric and $Q$ is self-adjoint and positive definite.
Putting $Y={\xi_i}$ in \eqref{2.2},
and using $Q\,{\xi_i}={\xi_i}$,
we get ${\eta^i}(X)=g(X,{\xi_i})$.
Hence, ${\xi_i}$ is orthogonal to ${\cal D}$ for any compatible metric.
Thus, $TM={\cal D}\oplus{\cal D}^\bot$ -- the sum of two complementary orthogonal subbundles.

\smallskip

A~distribution $\widetilde{\cal D}\subset TM$
is called {totally geodesic} if and only if
its second fundamental form vanishes, i.e., $\nabla_X Y+\nabla_Y X\in\widetilde{\cal D}$ for any vector fields $X,Y\in\widetilde{\cal D}$ --
this is the case when {any geodesic of $M$ that is tangent to $\widetilde{\cal D}$ at one point is tangent to $\widetilde{\cal D}$ at all its points},
e.g., \cite[Section~1.3.1]{Rov-Wa-2021}.
According to the Frobenius theorem, any involutive distribution
is tangent to (the leaves of) a foliation.
Any involutive and totally geodesic distribution is tangent to a totally geodesic~foliation.
A~foliation whose orthogonal distribution is totally geodesic is called a Riemannian foliation.


A ``small" (1,1)-tensor $\widetilde{Q} = Q - {\rm id}_{\,TM}$ measures the difference between weak and classical $f$-structures.
By \eqref{Eq-f-01}, we obtain
\[
 [\widetilde{Q},{f}]=0,\quad \widetilde{Q}\,{\xi_i}=0,\quad \eta^i\circ\widetilde{Q}=0.
\]

\begin{Proposition}[see \cite{rst-43}]\label{thm6.1}
The normality condition for a weak metric $f$-structure implies
\begin{align}\label{Eq-normal}
 & {\cal N}^{\,(3)}_i = {\cal N}^{\,(4)}_{ij} = 0,\quad {\cal N}^{\,(2)}_i(X,Y) = \eta^i([\widetilde QX, fY]), \\
\label{Eq-normal-2}
 & \nabla_{\xi_i}\,\xi_j\in{\cal D},\quad  [X,\xi_i]\in{\cal D}\quad
 (X\in{\cal D}).
\end{align}
Moreover, ${\cal D}^\bot$ is a totally geodesic distribution.
\end{Proposition}

The coboundary formula for exterior derivative of a $2$-form $\Phi$ is
\begin{align}\label{E-3.3}
 d\Phi(X,Y,Z) &= \frac{1}{3}\big\{
 X\,\Phi(Y,Z) + Y\,\Phi(Z,X) + Z\,\Phi(X,Y) \notag\\
 &-\Phi([X,Y],Z) - \Phi([Z,X],Y) - \Phi([Y,Z],X) \big\}.
\end{align}
Note that $d\Phi(X,Y,Z)=-\Phi([X,Y],Z)$ for $X,Y\in{\cal D}^\bot$.
Therefore, for a weak metric $f$-structure, the distribution ${\cal D}^\bot$ is involutive if and only if $d\Phi=0$.

Only one new tensor ${\cal N}^{(5)}$ (vani\-shing at $\widetilde Q=0$), which supplements the sequence of well-known tensors
${\cal N}^{\,(1)},{\cal N}^{\,(2)}_i,{\cal N}^{\,(3)}_i,{\cal N}^{\,(4)}_{ij}$
is needed to study the weak $f$-contact structure.

\begin{Proposition}[see \cite{rst-43}]
\label{lem6.0}
For a weak metric $f$-structure
we get
\begin{align}\label{3.1-new}
\notag
 & 2\,g((\nabla_{X}{f})Y,Z) = 3\,d\Phi(X,{f} Y,{f} Z) - 3\, d\Phi(X,Y,Z) + g({\cal N}^{\,(1)}(Y,Z),{f} X)\notag\\
\notag
 & +\sum\nolimits_{\,i}\big({\cal N}^{\,(2)}_i(Y,Z)\,\eta^i(X) + 2\,d\eta^i({f} Y,X)\,\eta^i(Z) - 2\,d\eta^i({f} Z,X)\,\eta^i(Y)\big) \notag\\
 & + {\cal N}^{\,(5)}(X,Y,Z),
\end{align}
where a skew-symmet\-ric with respect to $Y$ and $Z$ tensor ${\cal N}^{\,(5)}(X,Y,Z)$ is defined by
\begin{align*}
 & {\cal N}^{\,(5)}(X,Y,Z) = {f} Z\,(g(X, \widetilde QY)) - {f} Y\,(g(X, \widetilde QZ)) \\
 & + g([X, {f} Z], \widetilde QY) - g([X,{f} Y], \widetilde QZ) + g([Y,{f} Z] -[Z, {f} Y] - {f}[Y,Z],\ \widetilde Q X).
\end{align*}
For~particular values of the tensor ${\cal N}^{\,(5)}$ we get
\begin{align}\label{KK}
\nonumber
 {\cal N}^{\,(5)}(X,\xi_i,Z) & = -{\cal N}^{\,(5)}(X, Z, \xi_i) = g( {\cal N}^{\,(3)}_i(Z),\, \widetilde Q X),\\
\nonumber
 {\cal N}^{\,(5)}(\xi_i,Y,Z) &= g([\xi_i, {f} Z], \widetilde QY) -g([\xi_i,{f} Y], \widetilde QZ),\\
 {\cal N}^{\,(5)}(\xi_i,\xi_j,Y) &= {\cal N}^{\,(5)}(\xi_i,Y,\xi_j)=0.
\end{align}
\end{Proposition}

Similar to the classical case, we introduce broad classes of weak metric $f$-structures.

\begin{Definition}\rm
(i) A weak metric $f$-structure $({f},Q,\xi_i,\eta^i,g)$ is called a \textit{weak ${K}$-structure} if it is normal and
$d \Phi=0$. We define two subclasses of weak ${K}$-manifolds
as follows:

(ii) \textit{weak ${\cal C}$-manifolds} if $d\eta^i = 0$ for any $i$, and

(iii) \textit{weak ${\cal S}$-manifolds} if the following is valid:
\begin{align}\label{2.3}
 \Phi=d{\eta^1}=\ldots =d{\eta^s} .
\end{align}
Omitting the normality condition, we get the following: a weak metric $f$-structure is called

(i)~a \textit{weak almost ${\cal K}$-structure} if $d\Phi=0$;

(ii)~a \textit{weak almost ${\cal C}$-structure} if $\Phi$ and $\eta^i\ (1\le i\le s)$ are closed forms;

(iii)~a \textit{weak almost ${\cal S}$-structure} if \eqref{2.3} is valid (hence, $d\Phi=0$).

\noindent
A weak almost ${\cal S}$-structure, whose structure vector fields ${\xi_i}$ are Killing,~i.e.,
\begin{eqnarray*}
 (\pounds_{{\xi_i}}\,g)(X,Y)
  = g(\nabla_Y {\xi_i}, X) + g(\nabla_X {\xi_i}, Y) = 0 ,
\end{eqnarray*}
is called a \textit{weak $f$-{\rm K}-contact structure}.
\end{Definition}

For a weak almost ${\cal K}$-structure (and its special cases, a weak almost ${\cal S}$-structure and a weak almost ${\cal C}$-structure), the distribution ${\cal D}^\bot$ is involutive.
Moreover, for a weak almost ${\cal S}$-structure and a weak almost ${\cal C}$-structure, we obtain $[\xi_i, \xi_j]=0$;
in other words, the distribution ${\cal D}^\bot$ of these manifolds
is tangent to a $\mathfrak{g}$-foliation with an abelian Lie~algebra.

\begin{Remark}\rm
Let $\mathfrak{g}$ be a Lie algebra of dimension $s$.
We say that a foliation $\mathcal{F}$ of dimension $s$ on a smooth connected manifold $M$ is a $\mathfrak{g}$-\textit{foliation} if there exist
complete vector fields $\xi_1,\ldots,\xi_s$ on $M$ which, when restricted to each leaf of $\mathcal{F}$, form a parallelism of this submanifold
with a Lie algebra isomorphic to $\mathfrak{g}$, see, for example, \cite{AM-1995,RWo-2}.
\end{Remark}


The following diagram (well known for classical structures) summarizes the relationships between some classes of weak metric $f$-manifolds considered in this article:
\[
\left|
   \begin{array}{c}
  \textrm{weak\ metric} \\
    f\textrm{-manifold}\\
   \end{array}
 \right|
\overset{\Phi=d\eta^i}\supset
  \left|
   \begin{array}{c}
  \textrm{weak\ almost}  \\
    {\cal S}\textrm{-manifold} \\
   \end{array}
 \right|
\overset{\xi_i\,\textrm{-Killing}}\supset
  \left|
   \begin{array}{c}
  \textrm{weak}  \\
   f\textrm{-K-contact}\\
   \end{array}
 \right|
 \overset{{\cal N}^{\,(1)}=0}\supset
 \left|
   \begin{array}{c}
 \textrm{weak} \\
  {\cal S}\textrm{-manifold} \\
   \end{array}
 \right|.
\]
For $s=1$, we get the following diagram:
\[
\left|
   \begin{array}{c}
  \textrm{weak\ almost} \\
  \textrm{contact\ metric}\\
  \textrm{manifold}
   \end{array}
 \right|
\overset{\Phi=d\eta}\supset
  \left|
   \begin{array}{c}
  \textrm{weak\ contact} \\
  \textrm{metric} \\
  \textrm{manifold}
   \end{array}
 \right|
\overset{\xi\,\textrm{-Killing}}\supset
  \left|
   \begin{array}{c}
  \textrm{weak} \\
  \textrm{K-contact} \\
  \textrm{manifold}
   \end{array}
 \right|
\overset{{\cal N}^{\,(1)}=0}\supset
 \left|
   \begin{array}{c}
 \textrm{weak} \\
 \textrm{Sasakian} \\
 \textrm{manifold}
   \end{array}
 \right|.
\]

\section{Geometry of weak almost ${\cal S}$-manifolds}
\label{sec:03a}

For a weak almost ${\cal S}$-structure, the distribution ${\cal D}$ is not involutive, since we have
\[
 g([X, {f} X], {\xi_i})= 2\,d{\eta^i}({f} X,X) = g({f} X,{f} X)>0\quad (X\in{\cal D}\setminus\{0\}).
\]

\begin{Proposition}[see Theorem~2.2 in \cite{rst-43}]\label{thm6.2A}
For a weak almost ${\cal S}$-structure, the tensors ${\cal N}^{\,(2)}_i$ and ${\cal N}^{\,(4)}_{ij}$ vanish;
moreover, ${\cal N}^{\,(3)}_i$ vanishes if and only if $\,\xi_i$ is a Killing vector field.
\end{Proposition}

By ${\cal N}^{\,(4)}_{ij}=0$ we~have
 $g(X,\nabla_{\xi_i}\,\xi_j)+g(\nabla_{X}\,\xi_i,\,\xi_j)=0$
 for all
 $X\in\mathfrak{X}_M$.
Symmetrizing the above equality and using
$g(\xi_i,\, \xi_j)=\delta_{ij}$ yield
 $\nabla_{\xi_i}\,\xi_j+\nabla_{\xi_j}\,\xi_i = 0$.
From this and $[\xi_i, \xi_j]=0$ it follows that 
\begin{align}\label{2-xi}
 \nabla_{\xi_i}\,\xi_j = 0,\qquad 1\le i,j\le s.
\end{align}

\begin{Corollary}
\label{thm6.2}
For a weak almost ${\cal S}$-structure, the distribution ${\cal D}^\bot$ is tangent to 
a $\mathfrak{g}$-foliation with totally geodesic flat $($that is $R_{\xi_i,\xi_j}\,\xi_k=0)$ leaves.
\end{Corollary}

The following corollary of Propositions~\ref{lem6.0} and \ref{thm6.2A} generalizes well-known results with $Q={\rm id}_{\,TM}$, e.g., \cite[Proposition 1.4]{b1970}
and \cite[Proposition~2.1]{DIP-2001}.

\begin{Proposition}
\label{lem6.1}
For a weak almost ${\cal S}$-structure we get
\begin{align}\label{3.1A}
\notag
 2\,g((\nabla_{X}{f})Y,Z) & = g({\cal N}^{\,(1)}(Y,Z),{f} X) {+}2\,g(fX,fY)\,\bar\eta(Z) {-}2\,g(fX,fZ)\,\bar\eta(Y) \\
 & + {\cal N}^{\,(5)}(X,Y,Z),
\end{align}
where $\bar\eta=\sum\nolimits_{\,i}\eta^i$.
Taking $X=\xi_i$ in \eqref{3.1A}, we obtain
\begin{align}\label{3.1AA}
 2\,g((\nabla_{\xi_i}{f})Y,Z) &= {\cal N}^{\,(5)}(\xi_i,Y,Z) ,\quad 1\le i\le s.
\end{align}
\end{Proposition}

The tensor ${\cal N}^{\,(3)}_i$ is important for weak almost ${\cal S}$-manifolds, see Proposition~\ref{thm6.2A}.
Therefore,
we define the tensor field $h=(h_1,\ldots,h_s)$,~where
\begin{align*}
 h_i=\frac{1}{2}\, {\cal N}^{\,(3)}_i
 = \frac{1}{2}\,\pounds_{\xi_i}{f} .
\end{align*}
Using $[\xi_i, \xi_j]=0$ and $f\xi_j=0$, we obtain
$(\pounds_{\xi_i}{f})\xi_j=[\xi_i, f\xi_j] - f[\xi_i,\xi_j]=0$;
therefore, $h_i\,\xi_j = 0$ is true.
For $X\in{\cal D}$, using $[\xi_i, \xi_j]=0$, we derive:
\[
 0 = 2\Phi(\xi_i,X) = 2\eta^j(\xi_i,X)
 = g(\nabla_X\,\xi_j, \xi_i);
\]
therefore, $g(\nabla_X\,\xi_j, \xi_i)=0$ for all $X\in\mathfrak{X}_M$.
Next, we calculate
\begin{align}\label{4.2}
 (\pounds_{\xi_i}{f})Y
 = (\nabla_{\xi_i}{f})Y - \nabla_{{f} Y}\,\xi_i + {f}\nabla_Y\,\xi_i.
\end{align}
Thus, using $g((\nabla_{\xi_i}{f})Y, \xi_j)=0$, see \eqref{3.1AA} with $Z=\xi_j$, we obtain $\eta^j\circ h_i=0$ for all $1\le i,j\le s$:
\[
 (\eta^j\circ h_i)Y =
 g((\pounds_{\xi_i}{f})Y, \xi_j)
 = g((\nabla_{\xi_i}{f})Y, \xi_j)
 -g(\nabla_{{f} Y}\,\xi_i , \xi_j)
 +g({f}\nabla_Y\,\xi_i, \xi_j)
 = 0.
\]



For an almost ${\cal S}$-structure, the tensor $h_i$ is self-adjoint, trace-free and anti-commutes with $f$,
i.e., $h_i{f}+{f}\, h_i=0$, see \cite{DIP-2001}.
We generalize this result for a weak almost ${\cal S}$-structure.

\begin{Proposition}[see \cite{rst-43,Rov-splitting}]\label{P-4.1}\rm
For a weak almost ${\cal S}$-structure $(f,Q,\xi_i,\eta^i,g)$, the tensor $h_i$ and its conjugate tensor $h_i^*$ satisfy
\begin{eqnarray*}
 g((h_i-h_i^*)X, Y) &=& \frac{1}{2}\,{\cal N}^{\,(5)}(\xi_i, X, Y),\quad X,Y\in TM,\\
 h_i{f}+{f}\, h_i &=& -\frac12\,\pounds_{\xi_i}{Q} ,\\
 h_i{Q}-{Q}\,h_i &=& \frac12\,[\,f,\, \pounds_{\xi_i}{Q}\,] .
\end{eqnarray*}
\end{Proposition}

The following corollary generalizes the known property of almost ${\cal S}$-manifolds.

\begin{Corollary}
 Let a weak almost ${\cal S}$-manifold satisfy $\pounds_{\xi_i}{Q} =0$, then
 $\tr h_i=0$.
\end{Corollary}

\begin{proof}
If $h_i X=\lambda X$, then using $h_i f =- f h_i$
(by assumptions and Proposition~\ref{P-4.1}),
we get $h_i f X=-\lambda f X$.
Thus, if $\lambda$ is an eigenvalue of $h_i$, then $-\lambda$ is also an eigenvalue of $h_i$; hence, $\tr h_i=0$.
\end{proof}

\begin{Definition}[see \cite{rst-43}]\rm
Framed weak $f$-structures $({f},Q,{\xi_i},{\eta^i})$ and $({f}',Q',{\xi_i},{\eta^i})$
on a smooth manifold
are said to be \textit{homothetic}
if the following conditions:
\begin{subequations}
\begin{align}\label{Tran'}
  & {f} = \sqrt\lambda\ {f}', \\
  & Q\,|_{\,{\mathcal D}}=\lambda\,Q'|_{\,\mathcal D} ,
\end{align}
\end{subequations}
are valid for some real $\lambda>0$.
Weak metric $f$-structures
$({f},Q,{\xi_i},{\eta^i},g)$ and $({f}',Q',{\xi_i},{\eta^i},g')$
are said to be \textit{homothetic} if they satisfy conditions (\ref{Tran'},b) and
\begin{align}\label{Tran2'}
 g|_{\,{\mathcal D}} = \lambda^{-\frac12}\,g'|_{\,{\mathcal D}},\quad
 g({\xi_i},\,\cdot) = {g}'({\xi_i},\,\cdot) .
\end{align}
\end{Definition}

\begin{Proposition}[see \cite{rst-43}]
Let a framed weak $f$-structure $({f},Q,{\xi_i},{\eta^i})$ satisfy
 $Q\,|_{\,{\mathcal D}}=\lambda\,{\rm id}_{\mathcal D}$
for some real $\lambda>0$. Then
$({f}', {\xi_i}, {\eta^i})$ is a framed $f$-structure, where ${f}'$ is given by \eqref{Tran'}.
Moreover, if $({f},Q,{\xi_i},{\eta^i},g)$ is a weak almost ${\cal S}$-structure
and \eqref{Tran'} and \eqref{Tran2'} are valid,
then $({f}',{\xi_i},{\eta^i},{g}')$ is an almost ${\cal S}$-structure.
\end{Proposition}

Denote by ${\rm Ric}({X},{Y})={\rm trace}(Z\to R_{Z,X}Y)$ the Ricci tensor,
where $R_{X, Y} = (\nabla_X\nabla_Y -\nabla_Y\nabla_X -\nabla_{[X,Y]})Z$ is the curvature tensor.
The Ricci operator
is given by
 $g(\Ric^\sharp X,Y)={\rm Ric}(X,Y)$.
The~scalar curvature of $g$ is given by $r={\rm trace}_g \Ric$.

\begin{Remark}\rm
For almost ${\cal S}$-manifolds, we have, see~\cite[Proposition~2.6]{DIP-2001},
\begin{equation}\label{Eq-div-f}
 \Ric^\sharp(\xi_i)=\sum\nolimits_{\,i=1}^{\,2n}(\nabla_{e_i}\,f)\,e_i=2\,n\,\bar\xi .
\end{equation}
Can one generalize \eqref{Eq-div-f}
for weak almost ${\cal S}$-manifolds\,?
\end{Remark}

For a weak almost ${\cal S}$-manifold, the \textit{splitting tensor}
$C:{\cal D}^\bot\times {\cal D}\to{\cal D}$ is defined~by
\begin{equation}\label{E-conulC}
  C_\xi(X)=-(\nabla_{X}\, \xi)^\top\quad (X\in{\cal D},\ \ \xi\in{\cal D}^\bot,\ \ \|\xi\|=1),
\end{equation}
where $^\top:TM\to{\cal D}$ is the orthoprojector, see \cite{Rov-splitting}.
The splitting tensor is decomposed as
 $C_\xi = A_\xi + T_\xi$,
where the skew-symmetric operator ${T}^{\sharp}_\xi$
and the self-adjoint operator ${A}^{\sharp}_\xi$
are given by
\begin{equation}\label{E-CC**}
 g({T}^{\sharp}_\xi X,Y)=g({T}(X,Y),\xi),\quad
 g({A}^{\sharp}_\xi X,Y)=g({b}(X,Y),\xi),\quad
 X,Y\in{\cal D},
\end{equation}
and $T(X,Y) =\frac12(\nabla_X Y-\nabla_Y X)^\bot$ is the integrability tensor
and $b(X,Y) =\frac12(\nabla_X Y+\nabla_Y X)^\bot$ is the second fundamental form
of ${\cal D}$.

Since ${\cal D}^\bot$ defines a totally geodesic foliation, see Corollary~\ref{thm6.2},
then the distribution ${\cal D}$ is {totally geodesic} if and only if
$C_\xi$
is skew-symmetric,
and ${\cal D}$ is integrable if and only if the
tensor $C_\xi$ is self-adjoint.
Thus, $C_\xi\equiv0$ if and only if ${\cal D}$ is integrable and defines a totally geodesic foliation; in this case, by de Rham Decomposition Theorem,
the manifold splits (is locally the product of Riemannian manifolds defined by distributions ${\cal D}$ and ${\cal D}^\bot$), e.g. \cite{Rov-Wa-2021}.

\begin{Theorem}
The splitting tensor of a weak almost ${\cal S}$-manifold
has the following view:
\begin{equation*}
  C_{\xi_i} = f + Q^{-1} f h^*_i\quad (i=1,\ldots,s).
\end{equation*}
\end{Theorem}

The \textit{mixed scalar curvature} of an
 almost product manifold
$M({\cal D}_1,{\cal D}_2,g)$
is the function
\begin{equation*}
 {r}_{\rm mix} =\sum\nolimits_{a,i} g({R}_{E_a, {\cal E}_i}\,E_a, {\cal E}_i),
\end{equation*}
where $\{{\cal E}_i,\,E_a\}$ is an adapted orthonormal frame,
i.e., $\{E_a\}\subset{\cal D}_1$ and $\{{\cal E}_i\}\subset{\cal D}_2$.
Let $b_i$ and $H_i$ be the second fundamental form and the mean curvature vector, and $T_i$ be the integrability tensor of the distribution ${\cal D}_i$.
The following
formula:
\begin{equation}\label{E-PW}
 {r}_{\rm mix}={\rm div}(H_1 + H_2) -\|b_1\|^2-\|b_2\|^2+\|H_1\|^2 +\|H_2\|^2 +\|T_1\|^2 +\|T_2\|^2,
\end{equation}
has many applications in Riemannian,
K\"{a}hler and Sasakian geometries, see \cite{Rov-Wa-2021}.

\begin{Theorem}
For the weak almost ${\cal S}$-structure
on a closed manifold $M^{2n+s}$ with conditions
\begin{equation*}
 (a)~\pounds_{\xi_i}\,\widetilde{Q} =0,\qquad
 (b)~N^{(5)}(\xi_i,\,\cdot\,, \,\cdot) =0,
\end{equation*}
(satisfied by weak $f$-K-manifolds)
the following integral formula is true:
\begin{eqnarray}\label{E-k-PW}
 \int_M \Big\{\sum\nolimits_{\,i} \big({\rm Ric}(\xi_i,\xi_i)  + \|Q^{-1} f h_i\|^2
 - (\tr Q^{-1} f h_i)^2\,\big) - s\,\|f\|^2 \Big\}\,d\,{\rm vol} = 0.
\end{eqnarray}
\end{Theorem}

\begin{proof}
According to \eqref{E-PW}, set ${\cal D}_1={\cal D}$ and ${\cal D}_2={\cal D}^\bot$.
Then $b_2=H_2=T_2=0$ and
\[
 b_1(X,Y)=\sum\nolimits_{\,i} g(Q^{-1} f h_i X, Y)\,\xi_i,\ \
 H_1=\sum\nolimits_{\,i} \tr (Q^{-1} f h_i)\,\xi_i,\ \
 T_1(X,Y)= g(f X, Y)\,\overline\xi ,
\]
where $\overline\xi = \sum\nolimits_{\,i} \xi_i$.
For a weak almost ${\cal S}$-manifold we have
 ${r}_{\rm mix}=\sum\nolimits_{\,i} {\rm Ric}(\xi_i,\xi_i)$.
Thus, \eqref{E-k-PW} is the counterpart of \eqref{E-PW} integrated on a closed
manifold using the Divergence Theorem.
\end{proof}

\begin{Definition}[see \cite{rov-126}]
\rm
An even-dimensional Riemannian manifold $(\tilde M,\tilde g)$ equipped with a skew-symmetric $(1,1)$-tensor $J$
such that $J^2$ is negative-definite is called a \textit{weak Hermitian manifold}.
This manifold is called \textit{weak K\"{a}hlerian} if $\tilde{\nabla} J=0$, where $\tilde\nabla$ is the Levi-Civita connection of $\tilde g$.
\end{Definition}

An involutive distribution is \textit{regular} if every point of the manifold has a neighborhood such that any integral submanifold passing through the neighborhood passes through only once, see, for example,~\cite{Fitz-2011}.
The next theorem states that a compact manifold with a regular weak almost ${\cal S}$-structure is a principal torus bundle over a weak Hermitean manifold,
and we believe that its proof using Proposition~\ref{thm6.2A} is similar to the proof of \cite[Theorem~4.2]{Fitz-2011}.

\begin{Theorem}
Let $M^{2n+s}$ be a compact manifold
equipped with a regular weak almost ${\cal S}$-structure $(f,Q,\tilde\xi_i,\tilde\eta^i,\tilde g)$.
Then there exists a weak almost ${\cal S}$-structure $(f,Q,\xi_i,\eta^i,g)$ on $M$ for which the structure vector fields
$\xi_1,\ldots,\xi_s$ are the infinitesimal generators of a free and effective $\mathbb{T}^s$-action on~$M$.
Moreover, the quotient $N = M / \mathbb{T}^s$ is a smooth weak Hermitean manifold of dimension~$2n$.
\end{Theorem}

\section{Geometry of weak ${\cal S}$--manifolds and weak ${\cal C}$-manifolds}
\label{sec:3cs}

By direct calculation we get the following:
\begin{align}\label{3.9A}
 (\pounds_{\xi_i}\,\Phi)(X,Y) = (\pounds_{\xi_i}\,g)(X, {f}Y) + g(X,(\pounds_{\xi_i}{f})Y) .
\end{align}
The following result generalizes \cite[Theorem~1.1]{b1970}.

\begin{Theorem}
On a weak ${\cal K}$-manifold the structure vector fields $\xi_1,\ldots,\xi_s$ are Killing and
\begin{align}\label{6.1e}
 \nabla_{\xi_i}\,\xi_j  = 0,\quad 1\le i,j\le s ;
\end{align}
thus, the
distribution ${\cal D}^\bot$ is tangent to 
a totally geodesic Riemannian foliation with flat leaves.
\end{Theorem}

\begin{proof}
By Proposition~\ref{thm6.1}, ${\cal D}^\bot$ is totally geodesic
and ${\cal N}^{\,(3)}_i{=}\,\pounds_{\xi_i}{f}=0$.
Using $\iota_{\,{\xi_i}}\Phi=0$ and condition $d\Phi=0$ in the identity
 $\pounds_{\xi_i}=\iota_{\,{\xi_i}}\,d + d\,\iota_{\,{\xi_i}}$,
we get $\pounds_{\xi_i}\Phi=0$. Thus, from \eqref{3.9A} we obtain $(\pounds_{\xi_i}\,g)(X, {f}Y)=0$.
To show $\pounds_{\xi_i}\,g=0$, we will examine $(\pounds_{\xi_i}\,g)(fX, \xi_j)$ and $(\pounds_{\xi_i}\,g)(\xi_k, \xi_j)$.
Using $\pounds_{\xi_i}\,\eta^j =0$,
we get $(\pounds_{\xi_i}\,g)(fX, \xi_j)=(\pounds_{\xi_i}\,\eta^j)fX -g(fX, [\xi_i,\xi_j])=-g(fX, [\xi_i,\xi_j])=0$.
Next, using Proposition~\ref{thm6.1},
we get
\[
 (\pounds_{\xi_i}\,g)(\xi_k, \xi_j)= -g(\xi_i, \nabla_{\xi_k}\,\xi_j+\nabla_{\xi_j}\,\xi_k) = 0.
\]
Thus, $\xi_i$ is a Killing vector field, i.e., $\pounds_{\xi_i}\,g=0$.
From $d\Phi(X,\xi_i,\xi_j)=0$ and \eqref{E-3.3} we obtain $g([\xi_i,\xi_j], fX)=0$, i.e., $\ker f$ is integrable.
Finally, from this and Proposition~\ref{thm6.1}
we get \eqref{6.1e}; thus, the sectional curvature $K(\xi_i,\xi_j)$ vanishes.
\end{proof}


\begin{Proposition}
For a weak ${\cal S}$-structure $(f,Q,\xi_i,\eta^i,g)$ we get
\begin{align}\label{4.10}
\notag
 & g((\nabla_{X}{f})Y,Z) = g(QX,Y)\,\bar\eta(Z) - g(QX,Z)\,\bar\eta(Y) \notag\\
 & +\sum\nolimits_{\,j} \eta^j(X)\big(\bar\eta(Y)\,\eta^j(Z) - \eta^j(Y)\,\bar\eta(Z)\big)
 +\frac{1}{2}\, {\cal N}^{\,(5)}(X,Y,Z).
\end{align}
Moreover, $\xi_i$ are Killing vector fields and ${\cal D}^\bot$ is tangent to a Riemannian totally geodesic foliation.
\end{Proposition}

The following theorem claims the rigidity of the ${\cal S}$-structure.

\begin{Theorem}\label{T-4.1A}
A weak metric $f$-structure
is a weak ${\cal S}$-structure if and only if it is an ${\cal S}$-structure.
\end{Theorem}

\begin{proof}
Let $({f},Q,\xi_i,\eta^i,g)$ be a weak ${\cal S}$-structure.
Since ${\cal N}^{\,(1)}=0$, by Proposition~\ref{thm6.1}, we get ${\cal N}^{\,(3)}_i=0$.
By \eqref{KK}, we get ${\cal N}^{\,(5)}(\cdot\,,\xi_i,\cdot\,)={\cal N}^{\,(5)}(\cdot\,,\cdot\,,\xi_i)=0$.
Since ${f}$ is skew-symmetric, applying \eqref{4.10} with $Z=\xi_i$ in \eqref{4.NN} yields
\begin{align*}
 g( [{f},{f}](X,Y),\xi_i) =  g((\nabla_{{f} X}{f})Y, \xi_i) - g((\nabla_{{f} Y}{f})X, \xi_i) = - 2\,g(QX,{f} Y) .
\end{align*}
From this and ${\cal N}^{\,(1)}=0$
we get $g(\widetilde{Q} X, {f} Y)=0$ for all $X,Y\in \mathfrak{X}_M$; therefore, $\widetilde Q=0$.
\end{proof}

For $s=1$, from Theorem~\ref{T-4.1A} we have the following

\begin{Corollary}[see \cite{RovP-arxiv}]
A weak almost contact metric structure on $M^{2n+1}$ is weak Sasakian if and only if it is a Sasakian structure
(i.e., a normal weak contact metric structure) on $M^{2n+1}$.
\end{Corollary}

Next, we study a weak almost ${\cal C}$-manifolds.

\begin{Proposition}
For a weak
${\cal C}$-structure $({f},Q,\xi_i,\eta^i,g)$, we obtain
\begin{align}\label{6.1}
 2\,g((\nabla_{X}{f})Y,Z) &= {\cal N}^{\,(5)}(X,Y,Z).
\end{align}
\end{Proposition}

A ${\cal K}$-structure is a ${\cal C}$-structure if and only if ${f}$ is parallel, e.g., \cite[Theorem~1.5]{b1970}.
 The~following our theorem extends this result
 and characterizes weak ${\cal C}$-manifolds
 using the condition $\nabla{f}=0$.

\begin{Theorem}\label{thm6.2D}
A weak metric $f$-structure $({f},Q,\xi_i,\eta^i,g)$ with conditions $\nabla{f}=0$ and
\begin{align*}
 [\xi_i,\xi_j]^\bot =0,\quad 1\le i,j\le s,
\end{align*}
is a~weak ${\cal C}$-structure with the property ${\cal N}^{\,(5)}=0$.
\end{Theorem}

\begin{proof}
Using condition $\nabla{f}=0$, from \eqref{4.NN} we obtain $[{f},{f}]=0$.
Hence, from \eqref{2.6X} we get ${\cal N}^{\,(1)}(X,Y)=2\,\sum\nolimits_{\,i} d\eta^i(X,Y)\,\xi_i$,
and from \eqref{4.NN} with $S=f$ and $Y=\xi_i$ we obtain
\begin{align}\label{E-cond1}
 \nabla_{{f} X}\,\xi_i - {f}\,\nabla_{X}\,\xi_i = 0,\quad X\in \mathfrak{X}_M.
\end{align}
From \eqref{E-3.3}, we calculate
\[
 3\,d\Phi(X,Y,Z) = g((\nabla_{X}{f})Z, Y) + g((\nabla_{Y}{f})X,Z) + g((\nabla_{Z}{f})Y,X);
\]
hence, using condition $\nabla{f}=0$ again, we get $d\Phi=0$. Next,
\begin{align*}
 {\cal N}^{\,(2)}_i(Y,\xi_j) = -\eta^i([{f} Y,\xi_j]) = g(\xi_j, {f}\nabla_{\xi_i} Y) =0.
\end{align*}
Thus, setting $Z=\xi_j$ in Proposition~\ref{lem6.1} and using the condition $\nabla{f}=0$ and the properties
$d\Phi=0$, ${\cal N}^{\,(2)}_i(Y,\xi_j)=0$ and ${\cal N}^{\,(1)}(X,Y)=2\sum\nolimits_{\,i} d\eta^i(X,Y)\,\xi_i$, we find
 $0 = 2\,d\eta^i({f} Y, X) - {\cal N}^{\,(5)}(X,\xi_i, Y)$.
 By~\eqref{KK} and \eqref{E-cond1}, we get
\[
 {\cal N}^{\,(5)}(X,\xi_i, Y) = g([\xi_i,{f} Y] -{f}[\xi_i,Y],\, \widetilde Q X)
 = g(\nabla_{{f} Y}\,\xi_i - {f}\,\nabla_{Y}\,\xi_i,\, \widetilde Q X) = 0;
\]
hence, $d\eta^i({f} Y, X)=0$. From this and condition
\[
 g([\xi_i,\xi_j],\xi_k)=2\,d\eta^k(\xi_j, \xi_i)=0
\]
we get $d\eta^i=0$. By the above, ${\cal N}^{\,(1)}=0$.
Thus, $({f},Q,\xi_i,\eta^i,g)$ is a weak ${\cal C}$-structure.
Finally, from \eqref{6.1} and condition $\nabla{f}=0$ we get ${\cal N}^{\,(5)}=0$.
\end{proof}

\begin{Example}\rm
Let $M$ be a $2n$-dimensional smooth manifold and $J:TM\to TM$ an endomorphism of rank $2n$ such that $\nabla J=0$.
To construct a weak ${\cal C}$-structure on $M\times\mathbb{R}^s$ or $M\times T^s$, where
$T^s$ is an $s$-dimensional torus,
take any point $(x, t_1,\ldots,t_s)$ of either
space and set $\xi_i = (0, \partial/\partial t_i)$, $\eta^i =(0, dt_i)$~and
\[
 {f}(X, Y) = (J X, 0),\quad
 Q(X, Y) = (- J^{\,2} X,\, Y).
\]
where $X\in M_x$ and $Y\in\{\mathbb{R}^s_t, T^s_t\}$.
Then \eqref{2.1} holds and Theorem~\ref{thm6.2D} can be used.
\end{Example}

\section{Geometry of weak $f$-{K}-contact manifolds}
\label{sec:02}

Here, we characterize weak $f$-{K}-contact manifolds among all weak almost ${\cal S}$-manifolds
and find conditions under which a Riemannian manifold endowed with a set of unit Killing vector fields becomes a weak $f$-{K}-contact manifold.

\begin{Lemma}
For a weak $f$-{K}-contact manifold
we obtain ${\cal N}^{\,(1)}({\xi_i},\,\cdot)=0$ and
\begin{eqnarray*}
 {\cal N}^{\,(5)}({\xi_i},\,\cdot\,,\,\cdot) & = & {\cal N}^{\,(5)}(\,\cdot\,,\,{\xi_i},\,\cdot) =0,\\
 \pounds_{{\xi_i}}{Q} &=& \nabla_{\xi_i}{Q} = 0 ,\\
 \nabla_{\xi_i} {f} &=& 0.
\end{eqnarray*}
\end{Lemma}

Recall the following property of $f$-{K}-contact manifolds, see \cite{b1970,Goertsches-2}:
\begin{equation}\label{E-30}
 \nabla\,{\xi_i} = -{f},\quad 1\le i\le s .
\end{equation}

\begin{Theorem}
A weak almost ${\cal S}$-structure is weak $f$-{K}-contact if and only if \eqref{E-30} is true.
\end{Theorem}

The mapping $R_{\,{\xi_i}}: X \mapsto R_{X,\,{{\xi_i}}}\,{{\xi_i}}\ (\xi\in{\cal D}^\bot,\ \|\xi\|=1)$ is called the \textit{Jacobi operator} in the ${\xi_i}$-direction, e.g., \cite{Rov-Wa-2021}.
For a weak almost ${\cal S}$-manifold, by Proposition~\ref{lem6.1}, we get $R_{\,{\xi_i}}(X)\,\in{\cal D}$.

\begin{Theorem}\label{prop2.1b}
Let $(M^{2n+s},g)$ be a Riemannian manifold  with orthonormal Killing vector fields $\xi_i,\ldots,\xi_s$ such that $d\,\eta^1=\ldots=d\,\eta^s$ $($where $\eta^i$ is the 1-form dual to $\xi_i)$ and the Jacobi operators $R_{\,{\xi_i}}\ (i\le s)$ are positive definite on the distribution ${\cal D}=\bigcap_{\,i}\ker\eta^i$.
Then the manifold is weak $f$-{K}-contact, and its structural tensors are the following:
\[
 {f} = -\nabla\,{\xi_i},\quad
 Q X = R_{\xi_i}(X)\quad
 (X\in{\cal D},\ \ 1\le i\le s).
\]
\end{Theorem}

\begin{proof}
Since ${\xi_i}$ are Killing vector fields, we obtain the property \eqref{2.3}:
\[
 d{\eta^i}(X,Y) = (1/2)\,(g(\nabla_X\,{\xi_i}, Y) - g(\nabla_Y\,{\xi_i}, X)) = - g(\nabla_Y\,{\xi_i}, X) = g(X,{f} Y).
\]
Set $Q X = R_{\,{\xi_i}}(X)$ for some $i$ and all $X\in{\cal D}$.
Since ${\xi_i}$ is a unit Killing vector field, we get $\nabla_{\xi_i}{\xi_i}=0$ and
$\nabla_X\nabla_Y\,{\xi_i} - \nabla_{\nabla_X Y}\,{\xi_i} = R_{\,X,{\xi_i}} Y$, see \cite{YK-1985}.
Thus, ${f}\,{\xi_i}=0$ is true, and
\[
 {f}^2 Y = \nabla_{\nabla_Y\,{\xi_i}}\,{\xi_i} = R_{\,{\xi_i},Y}\,{\xi_i} =-R_{\xi_i}(Y)= -QY,\quad Y\in{\cal D}.
\]
By the conditions,
the tensor $Q$ is positive definite on the subbundle ${\cal D}$.
Therefore, the rank of $f$ restricted to $\cal D$ is $2\,n$.
Set $Q\,{\xi_i} = {\xi_i}\ (1\le i\le s)$. Thus, \eqref{2.1} and \eqref{2.2} are true.
\end{proof}

The sectional curvature of a plane containing unit vectors: $\xi\in{\cal D}^\bot$ and $X\in{\cal D}$ is called \textit{mixed sectional curvature}.
The mixed sectional curvature of an almost ${\cal S}$-manifold is a spacial case of {mixed sectional curvature} of almost product manifolds, for example,~\cite{Rov-Wa-2021}.
Note that the mixed sectional curvature of an $f$-{K}-contact manifold~is constant equal to 1.

\begin{Proposition}
A weak $f$-{K}-contact structure $({f},Q,{\xi_i},{\eta^i},g)$ of constant mixed sectional curvature,
satisfying $K({\xi_i},X)=\lambda>0$ for all ${X}\in{\cal D}$ and some $\lambda=const\in\mathbb{R}$,
is homothetic to an $f$-{K}-contact structure $({f}',{\xi_i},{\eta^i},g')$ after the transformation $(\ref{Tran'},b)$--\eqref{Tran2'}.
\end{Proposition}

\begin{Example}\rm
According to Theorem~\ref{prop2.1b}, we can search for examples of weak $f$-{K}-contact (not $f$-{K}-contact) manifolds can be found among Riemannian manifolds of positive sectional curvature admitting $s\ge1$ mutually orthogonal unit Killing vector fields.
Set $s=1$, and let $M^{n+1}$ be an
 ellipsoid with induced metric $g$ of~$\mathbb{R}^{2n+2}\ (n\ge1)$,
\[
 M = \Big\{(u_1,\ldots,u_{2n+2})\in\mathbb{R}^{2n+2}: \sum\nolimits_{\,i=1}^{n+1} u_i^2 + a\sum\nolimits_{\,i=n+2}^{2n+1} u_i^2 = 1\Big\},
\]
where $0<a=const\ne1$. The sectional curvature of $(M,g)$ is positive. It~follows that
\[
 \xi = (-u_2, u_1, \ldots , -u_{n+1}, u_{n}, -\sqrt a\,u_{n+3}, \sqrt a\,u_{n+2}, \ldots , -\sqrt a\,u_{2n+2}, \sqrt a\,u_{2n+1})
\]
is a Killing vector field on $\mathbb{R}^{2n+2}$, whose restriction to $M$ has unit length.
Since $M$ is invariant under the flow of $\xi$, then
$\xi$ is a unit Killing vector field on~$(M,g)$.
\end{Example}

Since $K(\xi_i,\xi_j)=0$ for a weak almost ${\cal S}$-manifold, the~Ricci curvature in the ${\xi_j}$-direction is given~by
\[
 {\rm Ric}({{\xi_j}},{{\xi_j}})=\sum\nolimits_{\,i=1}^{\,2n} g(R_{e_i,\,{{\xi_j}}}\,{{\xi_j}}, e_i),
\]
where $(e_i)$ is a local orthonormal basis of~${\cal D}$.
Next proposition generalizes some particular properties of $f$-{K}-contact manifolds for the case of weak $f$-{K}-contact manifolds.

\begin{Proposition}
Let $M^{2n+s}(f,Q,\xi_i,\eta^i,g)$ be a weak $f$-{K}-contact manifold, then for all $i,j$ we get
\begin{eqnarray}
\label{E-R0}
 && R_{\,{\xi_i},\,Y} = \nabla_Y{f}\quad (Y\in\mathfrak{X}_M), \\
\label{E-R1}
 && R_{\,{\xi_i},Y}\,{{\xi_j}} = {f}^2 Y\quad (Y\in\mathfrak{X}_M), \\
\label{Eq-Ric-f}
 && \Ric^\sharp {\xi_i} = \Div f ,\\
\label{E-R1b}
 && {\rm Ric}({{\xi_i}},{{\xi_j}}) =\tr Q > 0 .
\end{eqnarray}
\end{Proposition}

\begin{Proposition}\label{C-5.0}
There are no Einstein weak $f$-{K}-contact manifolds with $s>1$. 
\end{Proposition}

\begin{proof}
 A weak $f$-{K}-contact manifold with $s>1$ and
 $\xi'=\frac{\xi_1+\xi_2}{\sqrt 2}$, satisfies the following:
\begin{equation}\label{E-Einst}
  {\rm Ric}(\xi', \xi') = \frac12\sum\nolimits_{i,j=1}^2{\rm Ric}(\xi_i,\xi_j) \overset{\eqref{E-R1b}}= 2\tr Q.
\end{equation}
If the manifold is an Einstein manifold, then for the unit vector field $\xi'$ we obtain ${\rm Ric}(\xi',\xi')={\rm Ric}(\xi_i,\xi_i)=\tr Q$.
Comparing this with \eqref{E-Einst} yields $\tr Q=0$ - a contradiction.
\end{proof}

For a $f$-{K}-contact manifold, equations \eqref{Eq-Ric-f} and \eqref{Eq-div-f} give $\Ric^\sharp(\xi_i)=2\,n\,\bar\xi\ (1\le i\le s)$
and $\Ric(\xi_i,\xi_i)=2\,n$.

\begin{Proposition}\label{C-5.1}
For a weak $f$-{K}-contact manifold,
 the mixed sectional curvature is positive:
\begin{equation*}
 K(\xi_i,{X})=g(Q{X},{X})>0\quad ({X}\in{\cal D},\ \|{X}\|=1),
\end{equation*}
and the Ricci curvature satisfies the following:
${\rm Ric}({\xi}_i,{\xi}_j)>0$ for all $1\le i,j\le s$.
\end{Proposition}

\begin{proof}
From \eqref{E-R1}, we obtain $K(\xi,X)=g(fX,fX)>0$ for any unit vectors $\xi\in{\cal D}^\bot,\ X\in{\cal D}$.
Using \eqref{2.1} and non-singularity of $f$ on ${\cal D}$, from \eqref{E-R1b} we get
\[
 \Ric({{\xi}_j},{{\xi}_i}) = \tr Q = -\sum\nolimits_{\,p=1}^{2n} g(f^2 e_p,  e_p) = \sum\nolimits_{\,p=1}^{2n} g(f e_p, f e_p)>0,
\]
where $(e_p)$ is a local orthonormal frame of ${\cal D}$,
thus the second statement is valid.
\end{proof}

\begin{Theorem}
A weak $f$-{K}-contact manifold
with conditions $(\nabla\Ric)({\xi_i},\,\cdot)=0\ (1\le i\le s)$
and $\tr Q=const$ is an Einstein manifold and $s=1$.
\end{Theorem}

\begin{proof}
Differentiating \eqref{E-R1b} and using \eqref{E-30} and the conditions, we have
\[
 0 = \nabla_Y\,({\rm Ric}({{\xi_i}},{{\xi_i}})) = (\nabla_Y\,{\rm Ric})({{\xi_i}},{{\xi_i}}) +2\,{\rm Ric}(\nabla_Y\,{{\xi_i}},{{\xi_i}})
 = -2\,{\rm Ric}({f} Y,{{\xi_i}}),
\]
hence ${\rm Ric}(Y,{{\xi_i}}) = (\tr Q)\,{\eta^i}(Y)$.
Differentiating this, then using
\[
 X({\eta^i}(Y))=g(\nabla_X{\xi_i}, Y)=-g({f} X,Y)+g(\nabla_X Y, {\xi_i})
\]
and assuming $\nabla_X Y=0$ at a point $x\in M$, gives
\begin{align*}
 & (\tr Q)\,g({f} X, Y) = \nabla_Y\,({\rm Ric}(X,{{\xi_i}})) \\
 & = (\nabla_Y\,{\rm Ric})(X,{{\xi_i}}) +2\,{\rm Ric}(X, \nabla_Y\,{{\xi_i}}) = -2\,{\rm Ric}(X, {f} Y).
\end{align*}
Thus, ${\rm Ric}(X, {f} Y) = (\tr Q)\,g(X, {f} Y)$.
Hence,
 ${\rm Ric}(X,Y) = (\tr Q)\,g(X,Y)$
for any vectors $X,Y\in T_xM$; therefore, $(M,g)$ is
an Einstein manifold. By Proposition~\ref{C-5.0}, $s=1$.
\end{proof}

The {partial Ricci curvature tensor}, see \cite{RWo-2}, is self-adjoint and is given by
\[
 \Ric^\top(X)=\sum\nolimits_{\,i=1}^s ({R}_{\,X^\top,\,\xi_i}\,\xi_i)^\top.
\]
Note that $\Ric^\top=s\,{\rm id}_{\,\cal D}$ holds for $f$-{K}-contact manifolds.
For weak $f$-{K}-contact manifolds, the tensor $\Ric^\top$ is positive definite, see Proposition~\ref{C-5.1}.
In \cite{RWo-2}, applying the flow of metrics for a $\mathfrak{g}$-foliation, a deformation
of a weak almost ${\cal S}$-structure with positive partial Ricci curvature onto the
the classical structures of the same kind was constructed.

The next theorem, using Proposition~\ref{C-5.1} and the method of \cite{RWo-2}, shows that a weak $f$-{K}-contact manifold can be deformed into an $f$-{K}-contact manifold.

\begin{Theorem}
Let $M^{2n+s}(f_0,Q_{\,0},\xi_i,\eta^i,g_0)$ be a weak $f$-{K}-contact manifold.
Then there exist metrics $g_\tau\ (\tau\in\mathbb{R})$
such that each $(f_t,Q_{\,\tau},\xi_i,\eta^i,g_\tau)$ is a weak $f$-{K}-contact structure satisfying
\begin{equation}\label{E-Q-phi}
 Q_{\,\tau}=(1/s)\Ric^\top_\tau,\quad
 f_\tau\,|_{\,{\cal D}}=T_{\xi_i}^\sharp(\tau).
\end{equation}
Moreover,~$g_\tau$  converges exponentially fast, as $\tau\to-\infty$, to a metric $\hat g$
with $\Ric^\top_{\hat g}=s\,{\rm id}_{\,\cal D}$ that gives an $f$-{K}-contact structure on $M$.
\end{Theorem}

\begin{proof} For a weak $f$-K-contact manifold, the
tensor $\Ric^\top$ is positive definite.
Thus, we can apply the method of
\cite[Theorem~1]{RWo-2}.
Consider the partial Ricci  flow,
see \cite{Rov-Wa-2021},
\begin{equation}\label{E-GF-Rmix-Phi}
 \partial_\tau\,g_\tau = -2\,(\Ric^\top)^\flat_{g_\tau} +2\,s\,g^\top_\tau ,
\end{equation}
 where the tensor $g^\top$ is given by $g^\top(X,Y)=g(X^\top,Y^\top)$.
We obtain on ${\cal D}$ for $\tau=0$:
\[
 \Ric^\top = -\sum\nolimits_{\,i} (T_{\xi_i}^\sharp)^2
 = - s\,f^2
 = s\,Q ,
\]
and find $T_{\xi_i}^\sharp \Ric^\top =\Ric^\top T_{\xi_i}^\sharp$, see also \cite{RWo-2}. Thus,
 $\sum\nolimits_i T_{\xi_i}^\sharp \Ric^\top T_{\xi_i}^\sharp = -(\Ric^\top)^2$.
By~the above, we obtain the following ODE:
\[
 {\partial_\tau}\Ric^\top = 4\Ric^\top(\Ric^\top -\,s\,{\rm id}_{\,\cal D}).
\]
According to ODE theory, there exists a unique solution $\Ric^\top_\tau$ for $\tau\in\mathbb{R}$;
thus, a solution $g_\tau$ of \eqref{E-GF-Rmix-Phi} exists for $\tau\in\mathbb{R}$ and it is unique.
Observe that $(f(\tau),\xi_i,\eta^i,Q_\tau)$ with $f(\tau),Q_\tau$ given in \eqref{E-Q-phi}
is a weak $f$-K-contact structure on $(M,g_\tau)$.
By uniqueness of a solution, the flow \eqref{E-GF-Rmix-Phi} preserves the directions of eigenvectors of $\Ric^\top$,
and each eigenvalue $\mu_i>0$ satisfies the ODE
$\dot\mu_i=4\mu_i\,(\mu_i-s)$
with $\mu_i(0)>0$:
This ODE has the following solution:
\[
 \mu_i(\tau)= \frac{\mu_i(0)\,s}{\mu_i(0)+\exp(4\,s\,\tau)(s-\mu_i(0))}
\]
(a function $\mu_i(\tau)$ on $M$ for $\tau\in\mathbb{R}$)
with $\lim\limits_{\tau\to-\infty}\mu_i(\tau)=s$.
Thus, $\lim\limits_{\tau\to-\infty}\Ric^\top(\tau)=s\,{\rm id}_{\,\cal D}$.
Let~$\{e_i(\tau)\}$ be a $g_\tau$-orthonor\-mal frame of ${\cal D}$ of eigenvectors associated with $\mu_i(\tau)$,
we then get
 ${\partial_\tau} e_i = (\mu_i - s) e_i$.
Since $e_i(\tau)=z_i(\tau)\,e_i(0)$ with $z_i(0)=1$, then ${\partial_\tau} \log z_i(\tau) = \mu_i(\tau) - s$.
By the~above, we obtain $z_i(\tau) = (\mu_i(\tau)/\mu_i(0))^{1/4}$. Hence,
\[
 g_t(e_i(0),e_j(0))
 =z_i^{-1}(\tau)z_j^{-1}(\tau)\,g_t(e_i(\tau),e_j(\tau))
 =\delta_{ij}(\mu_i(0)\mu_j(0)/(\mu_i(\tau)\mu_j(\tau)))^{1/4}.
\]
As $\tau\to-\infty$, $g_\tau$ converges to metric $\hat g$ given on the frame $\{e_i(0)\}$ by the equalities
$\hat g(e_i(0),e_j(0))=\delta_{ij}\sqrt{\mu_i(0)/s}$.
\end{proof}


Denote by $\rho(n)-1$ the {maximal number of point-wise linearly independent vector fields on a~sphere}~$S^{n-1}$.
The~topological invariant $\rho(n)$, called the \textit{Adams number}, is
\[
 \rho(({\rm odd})\,2^{4b+c})=8b+2^c\ \ \mbox{for\ any\ integers} \ \ b\ge 0,\ 0\le c\le3,
\]
see Table~1, and the inequality $\rho(n)\le2\,\log_2n+2$ is valid, for example, \cite[Sect.~1.4.4]{Rov-Wa-2021}.

\bigskip

\centerline{Table 1. The number of vector fields on the $(2n-1)$-sphere.}
\smallskip
\begin{tabular}{|c|c|c|c|c|c|c|c|c|c|c|c|c|c|c|c|c|}
 \hline
  $2n-1$ & 1 & 3 & 5 & 7 & 9 & 11 & 13 & 15 & 17 & 19 & 21 & 23 & 25 & 27 & 29 \\
 \hline
  $\rho(2n)-1$ & 1 & 3 & 1 & 7 & 1 &\ 3 &\ 1 &\ 8 &\ 1 &\ 3 &\ 1 &\ 7 &\ 1 &\ 3 &\ 1 \\
 \hline
\end{tabular}

\bigskip\noindent
There are not many theorems in differential geometry that use $\rho(n)$.
Applying the Adams number, we obtain a topological obstruction to the existence of weak $f$-K-contact manifolds.

\begin{Theorem}
For a weak $f$-K-contact manifold $M^{2n+s}(f,Q,\xi_i,\eta^i,g)$
we have $s<\rho(2n)$.
\end{Theorem}

\begin{proof}
For a weak almost ${\cal S}$-structure, the following Riccati equation is true, e.g.,~\cite{Rov-Wa-2021}:
\begin{align}\label{Eq-Ricatti}
 \nabla_\xi\,C_\xi + (C_\xi)^2 + R_\xi = 0\quad (\xi\in{\cal D}^\bot).
\end{align}
Since the splitting tensor $C_\xi$ is skew-symmetric for a weak $f$-K-contact manifold, i.e., $C_\xi=T^{\sharp}_\xi$ and $A^{\sharp}_\xi=0$, see \eqref{E-CC**},
and the Jacobi operator $R_\xi$ is self-adjoined,
\eqref{Eq-Ricatti} reduces to two equations on ${\cal D}$:
\[
 \nabla_\xi\,T^{\sharp}_\xi=0\ {\rm (the\ skew\textrm{-}symmetric\ part)},\quad
 (T^{\sharp}_\xi)^2 = -R_\xi\  {\rm (the\ self\textrm{-}adjoint\ part)}.
\]
By this and Proposition~\ref{C-5.1}, we get $C_\xi(Y)\ne0$ for any $\xi\ne0$ and $Y\ne0$.
Note that a skew-symmetric linear operator $T^{\sharp}_\xi$ can only have zero real eigenva\-lues.
 Thus, for any point $x\in M$, the continuous vector fields
 $C_{\xi_i}(Y)\ (\|Y\|=1,\ 1 \le i \le s)$,
are tangent to the unit sphere $\mathbb{S}^{2n-1}_x\subset {\cal D}^\bot_x$.
If $s\ge\rho(2n)$, then these vector fields are linearly dependent at some point $\tilde Y\in\mathbb{S}^{2n-1}_x$
with weights $\lambda_i$, i.e., $\sum_i\lambda_i C_{\xi_i}(\tilde Y)=0$.
Then~the splitting tensor has real eigenvector: $C_\xi(\tilde Y)=\lambda\,\tilde Y$,
where $\xi=\sum_i\lambda_i\,\xi_i\ne0$ and
$\lambda=g(C_\xi(\tilde Y), \tilde Y)=0$, a~contradiction.
Thus, the inequality $s<\rho(2n)$~holds.
\end{proof}

\section{Weak $f$-{K}-contact structure equipped with a generalized Ricci soliton}
\label{sec:04}


The following three lemmas are used in the proof of Theorem~\ref{T-5.1} given below.

\begin{Lemma}[see Lemma 3.1 in \cite{G-D-2020}]
\label{L-5.1}
For a weak $f$-{K}-contact manifold
the following holds:
\[
 (\pounds_{{\xi_i}}(\pounds_{X}\,g))(Z,{\xi_i}) = g(X,Z) + g(\nabla_{\xi_i}\nabla_{\xi_i} X, Z) + Z g(\nabla_{\xi_i} X, {\xi_i})
\]
for any $1\le i\le s$ and all vector fields $X,Z$ such that $Z$ orthogonal to ${\cal D}^\bot$.
\end{Lemma}

\begin{proof} This uses
$\nabla_{\xi_i}{\xi_i}=0$ and \eqref{E-R1}.
\end{proof}

\begin{Lemma}[e.g., \cite{G-D-2020}]\label{L-5.2}
Let $\sigma$ be a smooth function on a Riemannian manifold $(M; g)$. Then  for any vector fields $\xi$ and $Z$ on $M$
we have the following:
\[
 \pounds_{{\xi}}(d\sigma\otimes d\sigma)(Z,{\xi}) = Z({\xi}(\sigma))\,{\xi}(\sigma) + Z(\sigma)\,{\xi}({\xi}(\sigma)) .
\]
\end{Lemma}

Recall the following property of $f$-{K}-contact manifolds:
\begin{equation}\label{E-K-Ric-X}
 \Ric(Y,{\xi}) = 0\quad (Y\in{\cal D},\ \xi\in{\cal D}^\bot).
\end{equation}

\begin{Lemma}[see \cite{G-D-2020}]
\label{L-5.3}
Let a weak $f$-{K}-contact manifold
satisfy \eqref{E-K-Ric-X} and admit the generalized gradient Ricci soliton structure.
Then
\[
 \nabla_{\xi_i} \nabla\sigma = (\lambda + c_2\tr Q)\,{\xi_i} - c_1{\xi_i}(\sigma)\,\nabla\sigma .
\]
\end{Lemma}

The following theorem generalizes
\cite[Theorem 3.1]{G-D-2020}.

\begin{Theorem}\label{T-5.1}
Let a weak $f$-{K}-contact manifold
with the properties $\tr Q=const$ and \eqref{E-K-Ric-X} satisfy the following generalized gradient Ricci soliton equation
with $c_1(\lambda + c_2\tr Q)\ne -1$:
\begin{equation}\label{E-gg-r-e}
 {\rm Hess}_{\,\sigma} - c_2\Ric = \lambda\,g -c_1 d\sigma\otimes d\sigma
\end{equation}
for some $\sigma\in C^\infty(M)$ and $c_1,c_2,\lambda\in\mathbb{R}$.
Then $\sigma=const$, and if $c_2\ne0$, then our manifold is an Einstein manifold and $s=1$.
\end{Theorem}

\begin{proof}
Set $Z\in{\cal D}$.
Using Lemma~\ref{L-5.1} with $X=\nabla f$, we obtain
\begin{equation}\label{E-G3.6}
 2\,(\pounds_{{\xi_i}}({\rm Hess}_{\,\sigma}))(Z,{\xi_i}) = Z(\sigma) + g(\nabla_{\xi_i}\nabla_{\xi_i}\nabla\sigma, Z)
 + Z g(\nabla_{\xi_i}\nabla\sigma, {\xi_i}).
\end{equation}
Using Lemma~\ref{L-5.3} in \eqref{E-G3.6} and the properties $\nabla_{\xi_i}\,{\xi_i}=0$ and $g({\xi_i},{\xi_i})=1$, yields
\begin{eqnarray}\label{E-G3.7}
\nonumber
 && 2\,(\pounds_{{\xi_i}}({\rm Hess}_{\,\sigma}))(Z,{\xi_i}) = Z(\sigma) + a\,g(\nabla_{\xi_i}\,{\xi_i}, Z) \\
\nonumber
 && -c_1 g(\nabla_{\xi_i} ({\xi_i}(\sigma)\nabla\sigma), Z)
 + a\,Z(g({\xi_i},{\xi_i})) - c_1 Z({\xi_i}(\sigma)^2) \\
 && = Z(\sigma) -c_1 g(\nabla_{\xi_i}({\xi_i}(\sigma)\nabla\sigma), Z) - c_1 Z({\xi_i}(\sigma)^2),
\end{eqnarray}
where $a=\lambda + c_2\tr Q$.
Using Lemma~\ref{L-5.3} with $Z\in{\cal D}$, from \eqref{E-G3.7} we deduce
\begin{equation}\label{E-G3.8}
 2\,(\pounds_{{\xi_i}}({\rm Hess}_{\,\sigma}))(Z,{\xi_i}) = Z(\sigma) -c_1 {\xi_i}({\xi_i}(\sigma))\,Z(\sigma) +c_1^2{\xi_i}(\sigma)^2 Z(\sigma) - c_1 Z({\xi_i}(\sigma)^2).
\end{equation}
Since ${\xi_i}$ is a Killing vector field, we obtain $\pounds_{{\xi_i}}\,g=0$. This implies $\pounds_{{\xi_i}}\Ric=0$.
Using the above fact and applying the Lie derivative to equation \eqref{E-gg-r-e}, gives
\begin{equation}\label{E-G3.9}
 2\,(\pounds_{{\xi_i}}({\rm Hess}_{\,\sigma}))(Z,{\xi_i}) = -2\,c_1(\pounds_{{\xi_i}}(d\sigma\otimes d\sigma))(Z,{\xi_i}).
\end{equation}
Using \eqref{E-G3.8}, \eqref{E-G3.9} and Lemma~\ref{L-5.2}, we obtain
\begin{equation}\label{E-G3.10}
 Z(\sigma)\big(1 + c_1 {\xi_i}({\xi_i}(\sigma)) + c_1^2\,{\xi_i}(\sigma)^2 \big) = 0.
\end{equation}
Using Lemma~\ref{L-5.3}, we get
\begin{equation}\label{E-G3.11}
 c_1 {\xi_i}({\xi_i}(\sigma)) = c_1\,{\xi_i}(g({\xi_i}, \nabla\sigma)) = c_1 g({\xi_i}, \nabla_{\xi_i}\nabla\sigma)
 = c_1 a - c_1^2\,{\xi_i}(\sigma)^2.
\end{equation}
Applying \eqref{E-G3.10} in \eqref{E-G3.11}, we obtain $Z(\sigma)(c_1 a+1)=0$.
This implies $Z(\sigma)=0$ provided by $c_1 a+1 \ne 0$.
Hence, $\nabla\sigma\in{\cal D}^\bot$.
Taking the covariant derivative of $\nabla\sigma = \sum_{\,i}{\xi_i}(\sigma)\,{\xi_i}$ and using \eqref{E-30} and $\bar\xi=\sum_{\,i} \xi_i$, yields
\[
 g(\nabla_X\,\nabla\sigma, Z)= \sum\nolimits_{\,i} X({\xi_i}(\sigma))\,{\eta^i}(Z) -{\bar\xi}(\sigma)\,g({\sigma} X, Z),\quad X,Z\in\mathfrak{X}_M .
\]
From this, by symmetry of ${\rm Hess}_{\,\sigma}$, i.e., $g(\nabla_X\,\nabla\sigma, Z)=g(\nabla_Z\,\nabla\sigma, X)$,
we obtain the equality ${\bar\xi}(\sigma)\,g({\sigma} X, Z)=0$.
For $Z={\sigma} X$ with some $X\ne0$, since $g({\sigma} X, {\sigma} X)>0$, we get ${\bar\xi}(\sigma)=0$.
Replacing $(\xi_i)$ with another orthonormal frame from ${\cal D}^\bot$ preserves the weak $f$-{K}-contact structure and allows reaching any direction ${\tilde\xi}$ in ${\cal D}^\bot$. So $\nabla\sigma=0$, i.e. $\sigma=const$. Thus, from \eqref{E-gg-r-e} and $c_2\ne0$ it follows the claim.
\end{proof}

Motivated by Proposition~\ref{C-5.0}, we consider
\textit{quasi Einstein manifolds}, defined
by
\[
 \Ric(X,Y) = a\,g(X,Y) + b\,\mu(X)\,\mu(Y),
\]
where $a$ and $b$ are nonzero real scalars, and $\mu$ is a 1-form of unit norm.
If $\mu$ is the differential of a function, then we get a \textit{gradient quasi Einstein manifold}.

The following theorem
generalizes (and uses) Theorem~\ref{T-5.1}.

\begin{Theorem}
Let a weak $f$-{K}-contact manifold
with the properties $\tr Q=const$ and \eqref{E-K-Ric-X}
satisfy the following generalized gradient Ricci soliton equation:
\begin{equation}\label{E-gg-r-e2}
 {\rm Hess}_{\,\sigma_1} - c_2\Ric =  \lambda\,g -c_1 d\sigma_2\otimes d\sigma_2
\end{equation}
with $c_1 a \ne -1$, where $a=\lambda + c_2\tr Q$.
Then $\tilde\sigma=\sigma_1 +c_1 a\,\sigma_2$ is constant
and \begin{equation}\label{E-gg-r-e2b}
 -c_1 a\,{\rm Hess}_{\,\sigma_2} = -c_1\,d\sigma_2\otimes d\sigma_2 + c_2\Ric + \lambda\,g .
\end{equation}
Furthermore,

1. if $c_1 a\ne0$, then \eqref{E-gg-r-e2b} reduces to
 ${\rm Hess}_{\,\sigma_2} = \frac1a\,d\sigma_2\otimes d\sigma_2 - \frac{c_2}{c_1 a}\,\Ric - \frac\lambda{c_1 a}\,g$.
By Theorem~\ref{T-5.1}, if $c_1 a\ne-1$, then $\sigma_2=const$; moreover, if $c_2\ne0$, then $(M,g)$ is an Einstein manifold and $s=1$.

2. if $a=0$ and $c_1\ne0$, then \eqref{E-gg-r-e2b} reduces to
 $0 = c_2\Ric -c_1\,d\sigma_2\otimes d\sigma_2 + \lambda\,g$.
If $c_2\ne0$ and $\sigma_2\ne const$, then we get a gradient quasi Einstein manifold.

3. if $c_1=0$, then \eqref{E-gg-r-e2b} reduces to $0 = c_2\Ric + \lambda\,g$, and for $c_2\ne0$ we obtain an Einstein manifold
and $s=1$.
\end{Theorem}

\begin{proof} Similarly to Lemma~\ref{L-5.3}, we get
\begin{equation}\label{E-gg-r-e2c}
 \nabla_{\xi_i} \nabla\sigma_1 = a\,\xi_i - c_1\xi_i(\sigma_2)\,\nabla\sigma_2 .
\end{equation}
Using \eqref{E-gg-r-e2c} and Lemmas~\ref{L-5.1} and \ref{L-5.2}, and slightly modifying the proof of Theorem~\ref{T-5.1},
we find that the vector field $\nabla\tilde\sigma$ belongs to ${\cal D}^\bot$, where $\tilde\sigma=\sigma_1 +c_1 a \sigma_2$.
As~in the proof of Theorem~\ref{T-5.1}, we get $d\tilde\sigma=0$, i.e.,
 $d\sigma_1 = -c_1 a\,d\sigma_2$.
Applying this in \eqref{E-gg-r-e2}, we obtain \eqref{E-gg-r-e2b}.
Finally, from \eqref{E-gg-r-e2b}, we obtain the required three cases specified in the theorem.
\end{proof}

\section{Compact weak $f$-{K}-contact manifold equipped with an $\eta$-Ricci soliton}
\label{sec:05}

The following concepts were introduced in \cite{rov-127,rst-57}.

\begin{Definition}
\rm
An \textit{$\eta$-Ricci soliton} is a weak metric $f$-manifold $M({f},Q,{\xi_i},{\eta^i},g)$ satisfying
\begin{align}\label{Eq-1.1}
 \frac12\,\pounds_V\,g + \Ric = \lambda\,g +\mu\sum\nolimits_{\,i}\eta^i\otimes\eta^i +(\lambda+\mu)\sum\nolimits_{\,i\ne j}\eta^i\otimes\eta^j
\end{align}
for some smooth vector field $V$ on $M$ and functions $\lambda, \mu\in C^\infty(M)$.
A weak metric $f$-manifold $M({f},Q,{\xi_i},{\eta^i},g)$ is said to be \textit{$\eta$-Einstein},~if
\begin{align}\label{Eq-2.10}
 \Ric = a\,g + b\sum\nolimits_{\,i}\eta^i\otimes\eta^i + (a+b)\sum\nolimits_{\,i\ne j}\eta^i\otimes\eta^j\ \
 \mbox{for some}\ \ a,b\in C^\infty(M).
\end{align}
\end{Definition}


\begin{Remark}\rm
For a Killing vector field $V$, e.g., $V=\xi_i$ or $V=\bar\xi$,
equation \eqref{Eq-1.1} reduces to \eqref{Eq-2.10}.
Taking the trace of \eqref{Eq-2.10}, gives the
scalar curvature $r = (2\,n + s)\,a + s\,b$.
For $s=1$ and $Q={\rm id}$, \eqref{Eq-2.10} and \eqref{Eq-1.1} give the well-known definitions:
from~\eqref{Eq-2.10} we get
an $\eta$-Einstein structure
$\Ric = a\,g + b\,\eta\otimes\eta$,
and \eqref{Eq-1.1} gives an $\eta$-Ricci soliton
$\frac12\,\pounds_V\,g + \Ric = \lambda\,g + \mu\,\eta\otimes\eta$ on an almost contact metric~manifold. 
\end{Remark}

We~will use the generalized Pohozaev-Schoen identity.

\begin{Lemma}[e.g., \cite{G-O-2013}]
 Let $E$ be a divergence free symmetric {\rm (0,2)}-tensor,
 $E^0=E-\frac1{n}\,(\tr_{\!g}\,E)\,g$, and $V$ a vector field on a compact Riemannian manifold $(M^n, g)$ without boundary. Then
\begin{equation}\label{Eq-3.1}
 \int_{\,M} V(\tr_{\!g}\,E)\,{\rm d}\vol = \frac n2\int_{\,M} g(E^0, \pounds_V\,g)\,{\rm d}\vol .
\end{equation}
\end{Lemma}

The~constancy of the scalar curvature $r$ of a Riemannian manifold is important for proving the triviality of compact generalized $\eta$-Ricci solitons, see \cite{G-2023}.
The next theorem extends this result for
 weak $f$-{\rm K}-contact manifolds.

\begin{Theorem}[see \cite{rov-127}]
Let $M^{2n+s}(f, \xi_i, \eta^i, Q, g)$ be a compact weak $f$-{\rm K}-contact manifold with $r=const$ and $\tr Q=const$.
Suppose that $(g,V,\lambda,\nu)$ represents an $\eta$-Ricci soliton. Then $M$ is an $\eta$-Einstein manifold.
\end{Theorem}

\begin{proof}
Taking the trace of \eqref{Eq-2.10}, yields
 $r = (2n + s)\,\alpha + s\,\beta$.
Using \eqref{E-R1b} in \eqref{Eq-2.10}, we get $\alpha + \beta = \tr Q$.
Solution of the linear system is
 $\alpha = \frac{r-s\tr Q}{2\,n}$ and $\beta = \frac{(2\,n+s)\tr Q - r}{2\,n}$.
Define a traceless
(0,2)-tensor
\begin{equation*}
 E := \Ric -\frac{r-s\tr Q}{2\,n}g +\frac{r-(2\,n+s)\tr Q}{2\,n}\sum\nolimits_i\eta^i\otimes\eta^i
 -\tr Q\sum\nolimits_{i\ne j}\eta^i\otimes\eta^j .
\end{equation*}
Using this, we express $\eta$-Ricci soliton \eqref{Eq-1.1} as
\begin{align}\label{Eq-3.5}
\nonumber
 \frac12\,{\cal L}_V\,g + E & = \big(\lambda -\frac{r-s\tr Q}{2\,n}\big)g
 +\big(\nu +\frac{r-(2n+s)\tr Q}{2\,n}\big)\sum\nolimits_i\eta^i\otimes\eta^i \\
 & +(\lambda+\nu-\tr Q)\sum\nolimits_{i\ne j}\eta^i\otimes\eta^j.
\end{align}
We get the equality
 $\Div_g E=0$.
Since
$E$ is traceless, we obtain $E^0=E$ and $V(\tr_{\!g} E)=0$.
Then, applying
\eqref{Eq-3.1}, gives
 $\int_{\,M} g(E, {\cal L}_V\,g)\,{\rm d}\vol =0$.
Using \eqref{Eq-3.5} and \eqref{E-R1b} in the above integral formula,
and using the following equalities:
$g(E,g)= \tr_{\!g} E=0$ and
$g(E, \eta^i\otimes\eta^j)= E(\xi_i, \xi_j)=0$,
we get
 $\int_{\,M} g(E, E)\,{\rm d}\vol =0$.
By this,
$E=0$ is true.
Therefore, our manifold is an $\eta$-{Einstein} manifold.
\end{proof}

\section{Geometry of weak $\beta$-Kenmotsu $f$-manifolds}
\label{sec:02-f-beta}

The next definition generalizes the notions of $\beta$-Kenmotsu manifolds and Kenmotsu $f$-manifolds ($\beta=1,\ s>1$), see \cite{ghosh2019ricci,SV-2016}, and weak $\beta$-Kenmotsu manifolds ($s=1$), see \cite{rov-126}.

\begin{Definition}
\rm
A normal (i.e., ${\cal N}^{\,(1)} = 0$) weak metric $f$-manifold
will be called a \textit{weak $\beta$-Kenmotsu $f$-manifold} (a \textit{weak Kenmotsu $f$-manifold} when $\beta\equiv1$) if
\begin{align}\label{2.3-f-beta}
 (\nabla_{X}\,{f})Y=\beta\{g({f} X, Y)\,\bar\xi -\bar\eta(Y){f} X\}\quad (X,Y\in\mathfrak{X}_M),
\end{align}
where
$\bar\xi=\sum\nolimits_{\,i}\xi_i$, $\bar\eta=\sum\nolimits_{\,i}\eta^i$,
and $\beta$ is a nonzero smooth function on $M$.
\end{Definition}

Note that $\bar\eta(\xi_i)=\eta^i(\bar\xi)=1$ and $\bar\eta(\bar\xi)=s$.
Taking $X=\xi_j$ in \eqref{2.3-f-beta} and using ${f}\,\xi_j=0$, we get $\nabla_{\xi_j}{f}=0$, which implies $f(\nabla_{\xi_i}\,\xi_j)=0$,
and so $\nabla_{\xi_i}\,\xi_j\in {\cal D}^\bot$. This and the 1st equality in \eqref{Eq-normal-2} give
\begin{align}\label{Eq-normal-3}
 \nabla_{\xi_i}\,\xi_j =0\quad  (1\le i,j\le s),
\end{align}
thus, ${\cal D}^\bot$ (of weak $\beta$-Kenmotsu $f$-manifolds) is tangent to a totally geodesic
$\mathfrak{g}$-foliation with an abelian Lie~algebra.

\begin{Lemma}
For a weak $\beta$-Kenmotsu $f$-manifold
the following formula holds:
\begin{align}\label{2.3b}
 \nabla_{X}\,\xi_i = \beta\{X -\sum\nolimits_{\,j}\eta^j(X)\,\xi_j\}\quad (1\le i\le s,\quad X\in\mathfrak{X}_M).
\end{align}
\end{Lemma}

\begin{proof}
Taking $Y=\xi_i$ in \eqref{2.3-f-beta} and using $g({f} X, \xi_i)=0$ and $\bar\eta(\xi_i)=1$, we get ${f}(\nabla_X\,\xi_i -\beta X)=0$.
Since ${f}$ is non-degenerate on ${\cal D}$ and has rank $2\,n$, we get $\nabla_X\,\xi_i -\beta X = \sum\nolimits_{\,p}c^p\,\xi_p$.
The inner~product with $\xi_j$ gives $g(\nabla_X\,\xi_i,\xi_j) = \beta\,g(X,\xi_j) - c^j$.
Using \eqref{Eq-normal-2} and \eqref{Eq-normal-3}, we find $g(\nabla_X\,\xi_i,\xi_j)=g(\nabla_{\xi_i}X,\xi_j)=0$; hence, $c^j=\beta\,\eta^j(X)$.
This proves~\eqref{2.3b}.
\end{proof}


\begin{Theorem}
A weak metric $f$-manifold
is a weak $\beta$-Kenmotsu $f$-mani\-fold
if and only if the following conditions are valid:
\begin{align}\label{Eq-almost-K}
 {\cal N}^{\,(1)} = 0,\quad
 d\eta^i = 0,\quad
 d\Phi = 2\,\beta\,\bar\eta\wedge\Phi,\quad
 {\cal N}^{\,(5)}(X,Y,Z)=2\,\beta\,\bar\eta(X) g(f Y, \widetilde Q Z).
\end{align}
\end{Theorem}

\begin{proof}
Using \eqref{2.3b}, we obtain
\begin{align}\label{2.4A}
 (\nabla_X\,\eta^i)Y = X g(\xi_i, Y) -g(\xi_i, \nabla_X\,Y)
 = g(\nabla_{X}\,\xi_i, Y)
 = \beta\{g(X,Y) -\sum\nolimits_{\,j}\eta^j(X)\,\eta^j(Y)\}
\end{align}
for all $X,Y\in\mathfrak{X}_M$.
By \eqref{2.4A}, $(\nabla_X\,\eta^i)Y=(\nabla_Y\,\eta^i)X$ is true.
Thus, for $X,Y\in {\cal D}$ we obtain
\[
 0 = (\nabla_X\,\eta^i)Y-(\nabla_Y\,\eta^i)X
 = -\beta\,g([X,Y], \xi_i)
\]
that means that the distribution ${\cal D}$ is involutive, or equivalently, $d\eta^i(X,Y)=0$ for all $i=1,\ldots,s$ and $X,Y\in {\cal D}$.
By this and ${\cal N}^{\,(4)}_{ij}=0$, see \eqref{Eq-normal},
we find
$ d\eta^i = 0$.
Using \eqref{2.3-f-beta} and \eqref{E-3.3},
we get
\[
 3\,d\Phi(X,Y,Z) = 2\,\beta\{
  \bar\eta(X) g({f}Z, Y)
 +\bar\eta(Y) g({f}X, Z)
 +\bar\eta(Z) g({f}Y, X)
 \}.
\]
We also have the following:
\[
 3(\bar\eta\wedge\Phi)(X,Y,Z) =
  \bar\eta(X) g({f}Z, Y)
 +\bar\eta(Y) g({f}X, Z)
 +\bar\eta(Z) g({f}Y, X).
\]
Thus,
$d\Phi = 2\,\beta\,\bar\eta\wedge\Phi$ is valid.
By \eqref{4.NN} with $S=f$, and \eqref{2.3-f-beta}, we get $[{f},{f}]=0$; thus ${\cal N}^{\,(1)} = 0$.
Finally, from \eqref{3.1-new}, using \eqref{2.1} and \eqref{2.2}, we~obtain
\begin{align*}
 & g((\nabla_{X}{f})Y,Z) - \frac12\,{\cal N}^{\,(5)}(X,Y,Z) = 3\,\beta\big\{ (\bar\eta\wedge\Phi)(X,fY,fZ) - (\bar\eta\wedge\Phi)(X,Y,Z) \big\} \\
 & = \beta\big\{ -\bar\eta(X) g(QZ, fY) + \bar\eta(X) g(Z, fY) - \bar\eta(Y) g(fX, Z) - \bar\eta(Z) g(X, fY)\big\} \\
 & = \beta\big\{ \bar\eta(Z) g(fX, Y)  - \bar\eta(Y) g(fX, Z) - \bar\eta(X) g(fY, \widetilde Q Z) \big\}.
\end{align*}
From this, using \eqref{2.3-f-beta}, we get ${\cal N}^{\,(5)}(X,Y,Z)=2\,\beta\,\bar\eta(X) g(f Y, \widetilde Q Z)$.

Conversely, using  \eqref{2.1} and \eqref{Eq-almost-K} in \eqref{3.1-new}, we obtain
\begin{align*}
 & 2\,g((\nabla_{X}{f})Y,Z) = 6\,\beta\,(\bar\eta\wedge\Phi)(X,{f} Y,{f} Z) - 6\,\beta\,(\bar\eta\wedge\Phi)(X,Y,Z)
  +2\,\beta\,\bar\eta(X) g(\widetilde Q f Y, Z) \\
 & = 2\beta\big\{ {-}\bar\eta(X) g(fY, QZ) {-} \bar\eta(X) g({f}Z, Y) {-}\bar\eta(Y) g({f}X, Z) {-}\bar\eta(Z) g({f}Y, X)
  {+}\bar\eta(X) g(f Y, \widetilde QZ) \big\} \\
 & = 2\,\beta\{g({f} X, Y)\,g(\bar\xi, Z) -\bar\eta(Y) g({f}X, Z)\}  ,
\end{align*}
thus \eqref{2.3-f-beta} is true.
\end{proof}

\begin{Theorem}
A weak $\beta$-Kenmotsu $f$-manifold is locally a twisted product $\mathbb{R}^s\times_\sigma\bar M$
$($a warped product when $X(\beta)=0$ for $X\in{\cal D})$, where $\bar M(\bar g, J)$ is a weak K\"{a}hler manifold.
\end{Theorem}

\begin{proof}
By \eqref{Eq-normal-3}, the distribution ${\cal D}^\bot$ is tangent to a totally geodesic foliation,
and by the second equality of \eqref{Eq-normal-2}, the distribution ${\cal D}$ is tangent to a foliation.
By \eqref{2.3b}, the splitting tensor \eqref{E-conulC}
is conformal: $C_{\xi_i}X
= -\beta X\ (X\in{\cal D})$. Hence, ${\cal D}$ is tangent to a totally umbilical foliation with the mean curvature vector $H=-\beta\,\bar\xi$.
By \cite[Theorem~1]{pr-1993}, our manifold is locally a twisted product.
If $X(\beta)=0\ (X\in{\cal D})$ is true, then we get locally a warped product, see \cite[Proposition~3]{pr-1993}.
By \eqref{2.2}, the (1,1)-tensor $J=f|_{\,\cal D}$ is skew-symmetric and $J^2$ is negative definite.
To~show $\bar\nabla J=0$, using \eqref{2.3-f-beta} we find
$(\bar\nabla_X J)Y=\pi_{2*}((\nabla_X{f})Y)=0$ for $X,Y\in{\cal D}$.
\end{proof}

\begin{Example}\rm
Let $\bar M(\bar g,J)$ be a weak K\"{a}hler manifold and
$\sigma=c\,e^{\,\beta\sum t_i}$ a function on Euclidean space $\mathbb{R}^s(t_1,\ldots,t_s)$, where $\beta,c$ are nonzero constants.
Then the warped product manifold $M=\mathbb{R}^s\times_\sigma\bar M$ has a weak metric $f$-structure which satisfies \eqref{2.3-f-beta}.
Using \eqref{4.NN} with $S=J$, for a weak K\"{a}hler manifold, we get $[{J},{J}]=0$; hence, ${\cal N}^{\,(1)}=0$ is true.
\end{Example}

\begin{Corollary}
A weak Kenmotsu $f$-manifold $M^{2n+s}(f,Q,\xi_i,\eta^i,g)$ is locally a warped product $\mathbb{R}^s\times_\sigma \bar M$,
where $\sigma=c\,e^{\,\sum t_i}\ (c=const\ne0)$ and $\bar M(\bar g,J)$ is a weak K\"{a}hler manifold.
\end{Corollary}

To simplify the calculations in the rest of the paper, we assume that $\beta=const$.

\begin{Proposition}
\label{Lem-32}
The following formulas hold for weak $\beta$-Kenmotsu $f$-manifolds with $\beta=const$:
\begin{align*}
 & R_{X, Y}\,\xi_i = \beta^2\big\{\bar\eta(X)Y-\bar\eta(Y)X +\sum\nolimits_{\,j}\big(\bar\eta(Y)\eta^j(X) - \bar\eta(X)\eta^j(Y)\big)\xi_j\big\}
 \quad (X,Y\in\mathfrak{X}_M),\\
 & \Ric^\sharp \xi_i = -2\,n\,\beta^2\bar\xi , \\
 & (\nabla_{\xi_i}\Ric^\sharp)X = - 2\,\beta\Ric^\sharp X
 - 4\,n\,\beta^3\big\{ s\big(X - \sum\nolimits_{\,j}\eta^j(X)\,\xi_j\big) + \bar\eta(X)\bar\xi\big\}
 \quad (X\in\mathfrak{X}_M),\\
 & \xi_i(r) = -2\,\beta\{r + 2\,s\,n(2\,n+1)\,\beta^2\}.
\end{align*}
\end{Proposition}

The following theorem generalizes \cite[Theorem~1]{ghosh2019ricci} with $\beta\equiv1$ and $Q={\rm id}$.

\begin{Theorem}
Let $M^{2n+s}({f},Q,\xi_i,\eta^i,g)$ be a weak $\beta$-Kenmotsu $f$-manifold satisfying $\beta=const$.
If~$\nabla_{\xi_i}\Ric^\sharp=0$, then $(M,g)$ is an $\eta$-Einstein manifold \eqref{Eq-2.10} of constant scalar curvature
$r=-2\,s\,n(2\,n+1)\,\beta^2$.
\end{Theorem}

\begin{proof}
By conditions and Proposition~\ref{Lem-32},
$\Ric^\sharp Y = 2\,n\,\beta^2\big\{ s(Y - \sum\nolimits_{\,j}\eta^j(Y)\,\xi_j) + \bar\eta(Y)\bar\xi\big\}$, thus $r=-2\,s\,n(2\,n+1)\,\beta^2$.
Since \eqref{Eq-2.10} with $a=2\,s\,n\,\beta^2$ and $b=2\,(1-s)\,n\,\beta^2$ holds, $(M,g)$ is $\eta$-Einstein.
\end{proof}

\section{$\eta$-Ricci solitons on weak $\beta$-Kenmotsu $f$-manifolds}
\label{sec:03-f-beta}

The following lemmas are used in the proof of Theorem~\ref{thm3.1A} given below.

\begin{Lemma}\label{lem3.3}
 Let $M^{2n+s}({f},Q,\xi_i,\eta^i,g)$ be a weak $\beta$-Kenmotsu $f$-manifold with $\beta=const$.
If $g$ represents an $\eta$-Ricci soliton \eqref{Eq-1.1}, then
$\lambda+\mu=-2\,n\,\beta^2$.
\end{Lemma}

\begin{proof} For a weak $\beta$-Kenmotsu $f$-manifold equipped with an $\eta$-Ricci soliton \eqref{Eq-1.1}, using \eqref{Eq-normal-2}, we get
\begin{align*}
 (\pounds_V\,g)(\xi_i,\xi_j)=g(\xi_i, [V,\xi_j])=0.
\end{align*}
Thus, using \eqref{Eq-1.1} in the Lie derivative of $g(\xi_i, \xi_j) = \delta_{ij}$, we obtain
 $\Ric(\xi_i,\xi_j) = \lambda +\mu$.
Finally, using the equality $\Ric(\xi_i,\xi_j)=-2\,n\,\beta^2$, see
Proposition~\ref{Lem-32}, we achieve the result.
\end{proof}

\begin{Lemma}
Let $M^{2n+s}({f},Q,\xi_i,\eta^i,g)$ be a weak $\beta$-Kenmotsu $f$-manifold with $\beta=const$.
If $g$ represents an $\eta$-Ricci soliton  \eqref{Eq-1.1}, then
 $(\pounds_V R)_{X,\xi_j}\xi_i = 0$ for all $i,j$.
\end{Lemma}

\begin{Lemma}
On an $\eta$-Einstein \eqref{Eq-2.10} weak $\beta$-Kenmotsu $f$-manifold
with $\beta=const$, we obtain
\begin{align}\label{3.16}
 {\rm Ric}^\sharp X = \Big( s\,\beta^2+\frac{r}{2\,n} \Big)X -\Big((2\,n+s)\,\beta^2 + \frac{r}{2\,n}\Big)\sum\nolimits_{\,i}\eta^j(X)\,\xi_j
 -2\,n\beta^2\sum\nolimits_{\,i\ne j} \eta^i(X)\,\xi_j.
\end{align}
\end{Lemma}

\begin{proof}
Tracing \eqref{Eq-2.10} gives $r=(2\,n+s)\,a+sb$.
Putting $X=Y=\xi_i$ in \eqref{Eq-2.10} and using Proposition~\ref{Lem-32},
yields $a+b=-2\,n\,\beta^2$. Thus,
 $a = s\,\beta^2+\frac{r}{2\,n}$
 and
 $b = -(2\,n+s)\,\beta^2 - \frac{r}{2\,n}$,
and \eqref{Eq-2.10} gives
\eqref{3.16}.
\end{proof}

Next, we consider an $\eta$-Einstein weak $\beta$-Kenmotsu $f$-manifold as an $\eta$-Ricci soliton.

\begin{Theorem}\label{thm3.1A}
Let $g$ represent an $\eta$-Ricci soliton \eqref{Eq-1.1} on a weak $\beta$-Kenmot\-su $f$-manifold
with $\dim M>3$ and $\beta=const$.
If the manifold is also $\eta$-Einstein \eqref{Eq-2.10}, then
$a=-2\,s\,n\beta^2$, $b=2(s-1)n\beta^2$, and
the scalar curvature is $r=-2\,s\,n(2\,n+1)\,\beta^2$.
\end{Theorem}

\begin{Definition}
\rm
 A vector field $V$ on a weak metric $f$-manifold
 is called a {\it contact vector field},
 if
there exists a function $\rho\in C^\infty(M)$ such that
\begin{align*}
 \pounds_{X}\eta^i=\rho\,\eta^i,
\end{align*}
and if $\rho=0$,
i.e., the flow of $X$ preserves the forms $\eta^i$,
then $V$ is
a \textit{strict contact vector field}.
\end{Definition}

We consider the interaction of a weak $\beta$-Kenmotsu $f$-structure with an $\eta$-Ricci soliton whose potential vector field $V$ is a contact vector field, or $V$ is collinear to $\bar\xi$.

\begin{Theorem}
Let $M^{2n+s}({f},Q,\xi_i,\eta^i,g)$ be a weak $\beta$-Kenmotsu $f$-manifold with $\dim M>3$ and $\beta=const$.
If $g$ represents an $\eta$-Ricci soliton \eqref{Eq-1.1} with a contact potential vector field $V$, then $V$ is strict contact
and the manifold is $\eta$-Einstein \eqref{Eq-2.10} with $a=-2\,s\,n\beta^2,\ b=2(s-1)n\beta^2$ of
constant scalar curvature $r=-2\,s\,n(2\,n+1)\beta^2$.
\end{Theorem}

\begin{Theorem}
Let $M^{2n+s}({f},Q,\xi_i,\eta^i,g)$ be a weak $\beta$-Kenmotsu $f$-manifold with $\dim M>3$ and $\beta=const$.
If~$g$ represents an $\eta$-Ricci soliton \eqref{Eq-1.1} with a
potential vector field $V$ collinear to $\bar\xi$: $V=\delta\,\bar\xi$ for a smooth function $\delta\ne0$ on $M$, then $\delta=const$
and the manifold is $\eta$-Einstein \eqref{Eq-2.10} with $a=-2\,s\,n\beta^2$ and $b=2(s-1)n\beta^2$ of constant scalar curvature $r=-2\,s\,n(2\,n+1)\beta^2$.
\end{Theorem}

\begin{proof} Using \eqref{2.3-f-beta} in the derivative of $V=\delta\,\bar\xi$,
yields
\[
 \nabla_X V = X(\delta)\,\bar\xi +\delta\,\beta(X-\sum\nolimits_{\,j}\eta^j(X)\,\xi_j),\quad X\in\mathfrak{X}_M.
\]
Using this and calculations
\begin{align*}
& (\pounds_{\delta\,\bar\xi}\,g)(X,Y)=\delta(\pounds_{\bar\xi}\,g)(X,Y) +X(\delta)\,\bar\eta(Y) +Y(\delta)\,\bar\eta(X),\\
& (\pounds_{\bar\xi}\,g)(X,Y)=2\,s\,\beta\{g(X,Y) -\sum\nolimits_{\,j}\eta^j(X)\,\eta^j(Y)\},
\end{align*}
we transform the $\eta$-Ricci soliton equation \eqref{Eq-1.1} into
\begin{align}\label{3.23}
\notag
 & 2\,{\rm Ric}(X,Y) = -X(\delta)\,\bar\eta(Y) -Y(\delta)\,\bar\eta(X) +2(\lambda-\delta\beta)\,g(X,Y) \\
 & +2(\delta\beta+\mu)\,\sum\nolimits_{\,j}\eta^j(X)\,\eta^j(Y)
 -4\,n\beta^2\,\sum\nolimits_{\,i\ne j}\eta^i(X)\,\eta^j(Y),\quad X,Y\in\mathfrak{X}_M.
\end{align}
Inserting $X=Y=\xi_i$ in \eqref{3.23} and using Proposition~\ref{Lem-32}
and $\lambda+\mu=-2\,n\,\beta^2$, see Lemma~\ref{lem3.3}, we get $\xi_i(\delta)=0$.
It~follows from \eqref{3.23} and Proposition~\ref{Lem-32} that $X(\delta)=0\ (X\in{\cal D})$. Thus $\delta$ is constant on $M$, and \eqref{3.23}~reads
\begin{align*}
 {\rm Ric}= (\lambda-\delta\beta)\,g +(\delta\beta+\mu)\sum\nolimits_{\,j}\eta^j\otimes\eta^j -2\,n\beta^2\,\sum\nolimits_{\,i\ne j}\eta^i\otimes\eta^j.
\end{align*}
This shows that $(M,g)$ is an $\eta$-Einstein manifold with $a=\lambda-\delta\beta$ and $b=\mu+\delta\beta$ in \eqref{Eq-2.10}.
Therefore, from Theorem~\ref{thm3.1A} we conclude that $\lambda=\delta\beta-2\,s\,n\beta^2$, $\mu=-\delta\beta+2(s-1)n\beta^2$,
and the scalar curvature of $(M,g)$ is $r=-2\,s\,n(2\,n+1)\beta^2$.
\end{proof}

\section{Conclusions and future directions}

This review paper demonstrates that the weak metric $f$-structure is a valuable tool for exploring various geometric properties on manifolds,
including Killing vector fields, totally geodesic foliations, twisted products, Ricci-type solitons, and Einstein-type metrics.
Several results for metric $f$-manifolds have been extended to manifolds with weak structures, providing new
applications.

In conclusion, we pose several open questions:
1. {Is the condition ``the mixed sectional curvature is positive" sufficient for a weak metric $f$-manifold to be weak $f$-{K}-contact}?
2. {Does a weak metric $f$-manifold of dimension greater than $3$ have some positive mixed sectional curvature}?
3. {Is a compact weak $f$-{K}-contact Einstein manifold an ${\cal S}$-manifold}?
4. {When is a given weak $f$-{K}-contact manifold a mapping torus (see \cite{Goertsches-2}) of a manifold of lower dimension}?
5. {When does a weak metric $f$-manifold equipped with a Ricci-type soliton structure,
carry a canonical (e.g.,
of constant sectional curvature or
Einstein-type) metric}?
6. Can Theorems 13--16 be extended to the case where $\beta$ is not constant?
These questions highlight the potential for further exploration and development in the field of weak metric $f$-manifolds, encouraging continued research and discovery.

We delegate to the future:

\noindent\quad
$\bullet$ the study of integral formulas, variational problems and extrinsic geometric flows for weak metric $f$-manifolds and their distinguished classes using the methodology of \cite{Rov-Wa-2021}.

\noindent\quad
$\bullet$ the study of weak nearly ${\cal S}$- and weak nearly ${\cal C}$- manifolds as well as weak nearly Kenmotsu $f$-manifolds, and the generalization to the case $s>1$ of my results on weak nearly Sasakian/cosymplec\-tic manifolds,~see a survey~\cite{rst-55}.

\noindent\quad
$\bullet$  the study of geometric inequalities
with the mutual curvature invariants and with Chen-type invariants
(and the case of equality in them) for submanifolds in weak metric $f$-manifolds
and in their distinguished classes using the methodology of~\cite{rov-112}.


%




\begin{thebibliography}{999}

\bibitem{AM-1995}
 Alekseevsky, D.; Michor, P. {Differential geometry of $\mathfrak{g}$-manifolds}, {\em Differential Geom. Appl}. {\bf 1995}, {\em 5}, 371--403.

\bibitem{Mikes-2016}
Baishya, K.K.; Chowdhury, P.R.; Mike\v{s}, J.;  Pe\v{s}ka, P. On almost generalized weakly symmetric Kenmotsu manifolds. {\em Acta Univ. Palacki. Olomuc., Fac. Rerum Nat., Math.} {\em 55}, No. 2 ({\bf 2016}), 5--15.


\bibitem{b1970}
 Blair, D.\,E. Geometry of manifolds with structural group $U(n)\times O(s)$,  {\em J. Diff. Geom.} {\em 4} ({\bf 1970}), 155--167.

\bibitem{BP-2016}
Brunetti, L.; Pastore, A.M. $S$-manifolds versus indefinite $S$-manifolds and local decomposition theorems.
{\em International Electronic J. of Geometry}, {\bf 2016}, {\em 9:1}, 1--8.

\bibitem{CFF-1990}
Cabrerizo, J.\,L.; Fern\'{a}ndez, L.\,M.; Fern\'{a}ndez, M. The curvature tensor fields on $f$-manifolds with complemented frames.
{\em An. Stiint. Univ. Al. I. Cuza Iasi}, {\em 36} ({\bf 1990}), 151--161.


\bibitem{DIP-2001}
 Duggal, K.\,L.; Ianus, S.; Pastore, A.\,M. Maps interchanging $f$-structures and their harmoni\-city.
 {\em Acta Appl. Math.} {\em 67} ({\bf 2001}), 91--115.

\bibitem{fip}
Falcitelli, M.; Ianus, S.; Pastore, A. \emph{Riemannian Submersions and Related Topics}, World Scientific, 2004.

\bibitem{Fitz-2011}
Fitzpatrick, S. On the geometry of almost ${\cal S}$-manifolds.
{\em ISRN Geometry}, Vol. 2011, Article ID 879042, 12 pages.

\bibitem{Goertsches-2}
Goertsches, O.; Loiudice, E. On the topology of metric $f$-{K}-contact manifolds, {\em Monatshefte f\"{u}r Mathematik}, {\em 192} ({\bf 2020}), 355--370.

\bibitem{gy-1970}
Goldberg, S.\,I. and Yano, K. On normal globally framed $f$-manifolds, {\em Tohoku Math.} J. {\em 22} ({\bf 1970}), 362--370.

\bibitem{G-D-2020}
Ghosh, G.; De, U.\,C. Generalized Ricci soliton on K-contact manifolds, {\em Math. Sci. Appl. E-Notes}, {\em 8} (2020), 165--169.

\bibitem{ghosh2019ricci}
Ghosh, A. Ricci soliton and Ricci almost soliton within the framework of Kenmotsu manifold,
{\em Carpathian Math. Publ.}, {\em 11}\,(1) ({\bf 2019}), 59--69.

\bibitem{G-2023}
Ghosh, A. K-contact and $(k,\mu)$-contact metric as a generalized $\eta$-Ricci soliton.
{\em Math. Slovaca}, {\em 73}, No. 1 ({\bf 2023}), 185--194.

\bibitem{G-O-2013}
Gover, A.\,R. Orsted, B. Universal principles for Kazdan-Warner and Pohozaev-Schoen type identities,
{\em Commun. Contemp. Math.} {\em 15}\,(4) ({\bf 2013}), Art. ID 1350002.

\bibitem{Mikes-Hui-2016}
Hui, S. K.; Mikes, J.; Mandal, P. Submanifolds of Kenmotsu manifolds and Ricci solitons.
{\em J. Tensor Soc. (N.S.)} {\em 10} ({\bf 2016}), 79--89.

\bibitem{kenmotsu1972class}
Kenmotsu, K. A class of almost contact Riemannian manifolds,
{\em T\^{o}hoku Math. J.}, {\em 24} ({\bf 1972}), 93--103.

\bibitem{L-1969}
Ludden, G.D. Submanifolds of manifolds with an $f$-structure.
{\em Kodai Math. Semin. Rep.} {\em 21} ({\bf 1969}), 160--166.

\bibitem{nav-1983}
Naveira, A. A classification of Riemannian almost product manifolds, {\em Rend. Math.}, {\em 3} ({\bf 1983}), 577--592.

\bibitem{pr-1993}
Ponge, R.; Reckziegel, H. {Twisted products in pseudo-Riemannian geometry}, {\em Geom. Dedicata},
{\em 48} ({\bf 1993}), 15--25

\bibitem{rst-43}
Rovenski, V. Metric structures that admit totally geodesic foliations. {\em J. Geometry} ({\bf 2023}) 114:32.

\bibitem{Rov-splitting}
Rovenski, V. On the splitting tensor of the weak $f$-contact structure. {\em Symmetry} {\bf 2023}, {\em 15}\,(6), 1215.

\bibitem{rov-127}
 Rovenski, V. Einstein-type metrics and Ricci-type solitons on weak $f$-{K}-contact manifolds, pp. 29--51.
In: Rovenski et al.
(eds) \textit{Differential Geometric Structures and Applications},
2023. Springer Proceedings in Mathematics and Statistics, 440. Springer, Cham.

\bibitem{rst-57}
Rovenski, V. $\eta$-Ricci solitons and $\eta$-Einstein metrics on weak $\beta$-Kenmotsu $f$-manifolds.
Preprint, {\bf 2024}, 12 pages, arXiv:2412.14125.

\bibitem{rov-112}
Rovenski, V. Geometric inequalities for a submanifold equipped with distributions. {\em Mathematics} {\bf 2022}, 10, 4741. https://doi.org/10.3390/math10244741

\bibitem{rst-55}
Rovenski, V. Weak almost contact structures: a survey. 
17 p. {\bf 2024}. arXiv:2408.13827.
Accepted in: {\em Facta Universitatis} (Ni\v{s}).

\bibitem{rov-121}
Rovenski, V. Geometry of a weak para-$f$-structure.
{\em U.P.B. Scientific Bulletin, Series A: Applied Math.
and Physics}, {\bf 2023}, {\em 85}\,(4), 11--20.

\bibitem{Rov-Wa-2021}
Rovenski, V.; Walczak, P.\,G. \emph{Extrinsic geometry of foliations}, Progress in Mathematics, vol.~339,  Birkh\"{a}user, Cham, 2021.

\bibitem{RWo-2}
Rovenski, V.; Wolak, R. New metric structures on $\mathfrak{g}$-foliations, {\em Indagationes Mathematicae}, {\em 33} ({\bf 2022}), 518--532

\bibitem{RovP-arxiv}
Rovenski, V.; Patra, D.\,S. On the rigidity of the Sasakian structure and characterization of cosymplectic manifolds,
{\em Differential Geometry and its Applications}, {\em 90}
({\bf 2023}) 102043.

\bibitem{rov-126}
Patra, D.S.; Rovenski, V.  Weak $\beta$-Kenmotsu manifolds and $\eta$-Ricci solitons, pp. 53--72.
 In: Rovenski et al.
 (eds) \textit{Differential Geometric Structures and Applications},
2023. Springer Proceedings in Mathematics and Statistics, 440. Springer, Cham.

\bibitem{SV-2016}
Sari, R.; Turgut Vanli, A. Generalized Kenmotsu manifolds,
{\em Communications in Mathematics and Applications}, {\em 7}, No. 4 ({\bf 2016}), 311--328.

\bibitem{Di-T-2006}
 Di Terlizzi, L. On the curvature of a generalization of contact metric manifolds, {\em Acta Math. Hung.} 110, No. 3 ({\bf 2006}), 225--239.

\bibitem{tpw-2014}
Di Terlizzi, L.; Pastore, A.M.; Wolak, R.
Harmonic and holomorphic vector fields on an $f$-manifold with parallelizable kernel.
{\em An. Stiint. Univ. Al. I. Cuza Iausi, Ser. Noua, Mat}. {\bf 2014}, {\em 60}, No. 1, 125--144.

\bibitem{yan}
 Yano, K. On a structure $f$ satisfying $f^3+f=0$, {Technical Report No. 12}, University of Washington, 1961.

\bibitem{YK-1985}
 Yano, K.; Kon, M. \emph{Structures on Manifolds}, Vol. 3 of Series in Pure Math. World Scientific Publ. Co., Singapore, 1985.

\end{thebibliography}
\end{document}